\documentclass[final,leqno]{siamltex}

\usepackage{amssymb,amsmath,color,psfrag}
\usepackage[draft]{graphicx} 
\usepackage[noend]{algorithm,algorithmic}
\usepackage{xspace}
\usepackage{graphicx,color,psfrag}
\usepackage{caption,subfigure}

\usepackage{amsmath,amssymb,mathrsfs,bm}
\usepackage{dsfont}

 \usepackage[sort,numbers]{natbib}

\newtheorem{remark}[theorem]{Remark}
\newtheorem{assumption}[theorem]{Assumption}
\newtheorem{example}[theorem]{Example}

\newcommand{\oprocendsymbol}{\hbox{$\bullet$}}
\newcommand{\oprocend}{\relax\ifmmode\else\unskip\hfill\fi\oprocendsymbol}

\newcommand{\longthmtitle}[1]{\mbox{}\textup{\textrm{(#1).}}}

\newcommand{\seminorm}[2]{\|#1\|_{\text{\scalebox{1}{$#2$}}}}
\newcommand{\distU}[1]{|#1|_{\text{\scalebox{0.7}{$\mathcal{U}$}}}}

\newcommand{\distset}[2]{|#1|_{\text{\scalebox{0.7}{$#2$}}}}

\newcommand{\norm}[1]{\|#1\|_{\text{\scalebox{1}{$2$}}}}

\newcommand{\psdfunc}[1]{\lyap_2(#1)}
\newcommand{\psdfuncalone}{\lyap_2}
\newcommand{\lyapone}{\lyap_1}
\newcommand{\lyaptwo}{\lyap_2}

\newcommand{\absolute}[1]{|#1|}
\newcommand{\normfrob}[1]{|#1|_{\text{\scalebox{0.8}{$\mathcal{F}$}}}}

\newcommand{\expectation}{\mathbb{E}}

\newcommand{\expectationarg}[1]{\expectation\big[#1\big]}

\newcommand{\probability}{\mathbb{P}}

\newcommand{\co}{\operatorname{conv}}
\newcommand{\epi}{\operatorname{epi}}
\newcommand{\hyp}{\operatorname{hyp}}

\renewcommand{\diag}{\operatorname{diag}}

\newcommand{\kernel}{\operatorname{\mathcal{N}}}

\newcommand{\lambdamax}{\operatorname{\lambda_{\text{max}}}}

\newcommand{\lambdamin}{\operatorname{\lambda_{\text{min}}}}

\DeclareMathOperator*{\esssup}{ess\,sup}

\newcommand{\trace}{\operatorname{trace}}

\newcommand{\tp}{^{\text{\scalebox{0.8}{$T$}}}}

\renewcommand{\theenumi}{(\roman{enumi})}

\newcommand{\identity}{\mathrm{I}}

\newcommand{\indicator}[1]{\mathbf{1}_{#1}}

\newcommand{\lap}{\mathsf{L}}

\newcommand{\xu}{x_{\text{\scalebox{0.8}{$\Uset$}}}}
\newcommand{\xuperpendicular}{x_{\text{\scalebox{0.8}{$\Uset^{\bot}$}}}}

\newcommand{\hipartialset}{\{\operatorname{P}$i$\}_{i=1}^2}
\newcommand{\hiset}{\{\operatorname{P}$i$\}_{i=1}^3}
\newcommand{\hisettilde}{\{\operatorname{P}$i$\}_{i=1}^3}
\newcommand{\hisetfull}{\{\operatorname{P}$i$\}_{i=1}^4}
\newcommand{\hisetfulltilde}{\{\operatorname{P}$i$\}_{i=1}^4}

\newcommand{\hzero}{\operatorname{P0}}
\newcommand{\hone}{\operatorname{P1}}
\newcommand{\htwo}{\operatorname{P2}}
\newcommand{\hthree}{\operatorname{P3}}
\newcommand{\hfour}{\operatorname{P4}}

\newcommand{\dt}{\mathrm{d}t}
\newcommand{\ds}{\mathrm{d}s}

\newcommand{\relationconstants}{\thicksim}
\newcommand{\relation}{\thicksim^*}
\newcommand{\proper}{\vartriangleleft^{\text{\scalebox{0.8}{$\mathcal{K}$}}}}

\newcommand{\properinfty}{\thicksim^{\text{\scalebox{0.8}{$\mathcal{K}_{\infty}$}}}}
\newcommand{\properinftyccp}{\thicksim^{\text{\scalebox{0.8}{$\mathcal{K}_{\infty}^{cc}$}}}}

\newcommand{\lyapu}{\mathrm{U}}

\newcommand{\lyap}{\mathrm{V}}

\newcommand{\lyapw}{\mathrm{W}}

\newcommand{\Uset}{\mathcal{U}}

\newcommand{\muISS}{\mu}
\newcommand{\muISStilde}{\tilde{\mu}}
\newcommand{\increasingw}{\eta}
\newcommand{\gainnoise}{\sigma}

\newcommand{\gainNSSone}{\alpha_1}
\newcommand{\gainNSStwo}{\alpha_2}

\newcommand{\lyapbar}{\bar{\mathrm{V}}}

\newcommand{\lyapwbar}{\bar{\mathrm{W}}}

\newcommand{\lyaphat}{\hat{\mathrm{V}}}
\newcommand{\lyapwhat}{\hat{\mathrm{W}}}

\newcommand{\lyapudot}{\dot{\mathrm{U}}}

\newcommand{\convexf}{h}

\renewcommand{\grad}{\nabla}
\newcommand{\hessian}{\nabla^2}

\newcommand{\Input}{d}

\newcommand{\brownian}{\mathrm{B}}
\newcommand{\db}{\mathrm{dB}}

\newcommand{\fprocesssolution}{\{f(x(t),t)\}_{t\ge t_0}}

\newcommand{\gprocesssolution}{\{\galone(x(t),t)\}_{t\ge t_0}}

\newcommand{\fprocess}{\{f(t)\}_{t\ge t_0}}

\newcommand{\xprocess}{\{x(t)\}_{t\ge t_0}}

\newcommand{\xbarprocess}{\{\bar{x}(t)\}_{t\ge t_0}}

\newcommand{\borel}{\mathcal{B}}

\newcommand{\sigmaalgebra}{\mathcal{F}}
\newcommand{\filtrations}{\mathcal{F}_s}
\newcommand{\filtrationt}{\mathcal{F}_t}
\newcommand{\filtration}{\{\mathcal{F}_t\}_{t\ge0}}

\newcommand{\filtrationadapted}{\{\mathcal{F}_t\}}

\newcommand{\differential}{\mathrm{d}}

\newcommand{\ito}[2]{\mathcal{L}[#1](#2)}
\newcommand{\itowoarguments}[1]{\mathcal{L}[#1]}

\newcommand{\galone}{G}

\newcommand{\gs}{G(x(s),s)}
\newcommand{\gtx}{G(x,t)}

\newcommand{\gty}{G(y,t)}

\newcommand{\Atilde}{\tilde{A}}

\newcommand{\At}{A\tp}

\newcommand{\Bt}{B\tp}

\newcommand{\Obig}{\mathcal{O}}

\newcommand{\phip}{\phi_{\text{\scalebox{0.8}{$p$}}}}
\newcommand{\psip}{\psi_{\text{\scalebox{0.8}{$p$}}}}
\newcommand{\phiq}{\phi_{\text{\scalebox{0.8}{$q$}}}}
\newcommand{\psiq}{\psi_{\text{\scalebox{0.8}{$q$}}}}
\newcommand{\phione}{\phi_{\text{\scalebox{0.8}{$1$}}}}
\newcommand{\psione}{\psi_{\text{\scalebox{0.8}{$1$}}}}
\newcommand{\psionetilde}{\hat{\psi}_{\text{\scalebox{0.8}{$1$}}}}

\newcommand{\lone}{\mathcal{L}^1([t_0,\infty);\real^n)}

\newcommand{\ltwomatrices}{\mathcal{L}^2([t_0,\infty);\real^{n \times m})}

\newcommand{\dom}{\mathcal{D}}

\newcommand{\psddomain}{\mathcal{C}(\mathcal{D}\,;\real_{{\text{\scalebox{0.7}{$\ge 0$}}}})}
\newcommand{\psddomaintwice}{\mathcal{C}^{2}(\mathcal{D}\,;\real_{{\text{\scalebox{0.7}{$\ge 0$}}}})}

\newcommand{\psdtwice}{\mathcal{C}^{2}(\real^n;\real_{{\text{\scalebox{0.7}{$\ge 0$}}}})}

\newcommand{\psd}{\mathcal{C}(\real^n;\real_{{\text{\scalebox{0.7}{$\ge 0$}}}})}

\newcommand{\classk}{\mathcal{K}}
\newcommand{\classkl}{\mathcal{KL}}

\newcommand{\classkinfty}{\mathcal{K}_{\infty}}
\newcommand{\classkinftyccp}{\mathcal{K}_{\infty}^{cc}}

\newcommand{\realnm}{{\mathbb{R}^{n\times m}}}
\newcommand{\realmatrices}{{\mathbb{R}^{n\times n}}}

\newcommand{\realmatricesrectangulararg}[2]{\mathbb{R}^{#1\times #2}}

\newcommand{\real}{{\mathbb{R}}}

\newcommand{\realnonnegative}{{\mathbb{R}}_{\ge 0}}

\newcommand{\eps}{\epsilon}

\newcommand{\map}[3]{#1:#2 \rightarrow #3}

\newcommand{\setdef}[2]{\{#1 \, : \, #2\}}

\title{$p$th moment noise-to-state stability of stochastic
  differential equations with persistent noise\footnotemark[1]}
  
\author{David Mateos-N\'u\~nez \and Jorge Cort\'es\footnotemark[2]}

\begin{document}
\maketitle \renewcommand{\thefootnote}{\fnsymbol{footnote}}


\footnotetext[1]{A preliminary version of this manuscript was
  presented as~\cite{DMN-JC:13-acc} at the 2013 American Control
  Conference, Washington, D.C. This manuscript is a revision of the version submitted to
  SIAM Journal on Control and Optimization 52 (4) (2014), 2399-2421.}
 
\footnotetext[2]{The authors are with the Department of Mechanical and
  Aerospace Engineering, University of California, San Diego, 9500
  Gilman Dr, La Jolla, CA 92093, USA, {\tt \small
    \{dmateosn,cortes\}@ucsd.edu}.}

\renewcommand{\thefootnote}{\arabic{footnote}}

\begin{abstract}
  This paper studies the stability properties of stochastic
  differential equations subject to persistent noise (including the
  case of additive noise), which is noise that is present even at the
  equilibria of the underlying differential equation and does not
  decay with time.  The class of systems we consider exhibit
  disturbance attenuation outside a closed, not necessarily bounded,
  set.  We identify conditions, based on the existence of Lyapunov
  functions, to establish the noise-to-state stability in probability
  and in \textit{p}th~moment of the system with respect to a closed
  set.  As part of our analysis, we study the concept of two functions
  being proper with respect to each other formalized via pair of
  inequalities with comparison functions. We show that such
  inequalities define several equivalence relations for increasingly
  strong refinements on the comparison functions.  We also provide a
  complete characterization of the properties that a pair of functions
  must satisfy to belong to the same equivalence class.  This
  characterization allows us to provide checkable conditions to
  determine whether a function satisfies the requirements to be a
  strong NSS-Lyapunov function in probability or a $p$th~moment
  NSS-Lyapunov function. Several examples illustrate our results.
\end{abstract}

\section{Introduction}\label{sec:intro}

Stochastic differential equations (SDEs) go beyond ordinary
differential equations (ODEs) to deal with systems subject to
stochastic perturbations of a particular type, known as white noise.
Applications are numerous and include option pricing in the stock
market, networked systems with noisy communication channels, and, in
general, scenarios whose complexity cannot be captured by
deterministic models. In this paper, we study SDEs subject to
\emph{persistent noise} (including the case of additive noise), i.e.,
systems for which the noise is present even at the equilibria of the
underlying ODE and does not decay with time. Such scenarios arise, for
instance, in control-affine systems when the input is corrupted by
persistent noise.  For these systems, the presence of persistent noise
makes it impossible to establish in general a stochastic notion of
asymptotic stability for the (possibly unbounded) set of equilibria of
the underlying ODE.  Our aim here is to develop notions and tools to
study the stability properties of these systems and provide
probabilistic guarantees on the size of the state of the system.

\emph{Literature review:} In general, it is not possible to obtain
explicit descriptions of the solutions of SDEs.  Fortunately, the
Lyapunov techniques used to study the qualitative behavior of
ODEs~\cite{HKK:02,RFB:91} can be adapted to study the stability
properties of SDEs as well~\cite{RK:12,UHT:97,XM:99}.  Depending on
the notion of stochastic convergence used, there are several types of
stability results in SDEs.  These include stochastic stability (or
stability in probability), stochastic asymptotic stability, almost
sure exponential stability, and \textit{p}th moment asymptotic
stability, see e.g.,~\cite{UHT:97,XM:99,XM:11,TT:03}.
However, these notions are not appropriate in the presence of
persistent noise because they require the effect of the noise on the
set of equilibria to either vanish or decay with time.
To deal with persistent noise, as well as other system properties like
delays, a concept of ultimate boundedness is required that generalizes
the notion of convergence. As an example, for stochastic delay
differential equations,~\cite{FW-PEK:13} considers a notion of
ultimate bound in $p$th moment~\cite{HS:01} and employs Lyapunov techniques to
establish it.
More generally, for mean-square random dynamical systems, the concept
of forward attractor~\cite{PEK-TL:12} describes a notion of
convergence to a dynamic neighborhood and employs contraction analysis to establish it.
Similar notions of ultimate boundedness for the state of a system, now
in terms of the size of the disturbance, are also used for
differential equations, and
%
%
%
many of these notions are inspired by dissipativity properties of the
system that are captured via \emph{dissipation inequalities} of a
suitable Lyapunov function:
such inequalities encode the fact that the Lyapunov function decreases
along the trajectories of the system as long as the state is ``big
enough'' with regards to the disturbance.  As an example, the concept
of input-to-state stability (ISS) goes hand in hand with the concept
of ISS-Lyapunov function, since the existence of the second implies
the former (and, in many cases, a converse result is also
true~\cite{EDS-YW:95}). Along these lines, the notion of practical
stochastic input-to-state stability (SISS)
in~\cite{SL-JZ-ZJ:08,ZJW-XJX-SYZ:07} generalizes the concept of ISS to
SDEs where the disturbance or input affects both the deterministic
term of the dynamics and the diffusion term modeling the role of the
noise.  Under this notion, the state bound is guaranteed in
probability, and also depends, as in the case of ISS, on a decaying
effect of the initial condition plus an increasing function of the sum
of the size of the input and a positive constant related to the
persistent noise.
%
For systems where the input modulates the covariance of the noise,
SISS corresponds to noise-to-state-stability (NSS)~\cite{HD-MK:00},
which guarantees, in probability, an ultimate bound for the state that
depends on the magnitude of the noise covariance.  That is, the noise
in this case plays the main role, since the covariance can be
modulated arbitrarily and can be unknown.  This is the appropriate
notion of stability for the class of SDEs with persistent
noise considered in this paper, which are nonlinear systems affine in
the input, where the input corresponds to white noise with locally
bounded covariance. Such systems cannot be studied under the ISS
umbrella, because the stochastic integral against Brownian motion has
infinite variation, whereas the integral of a legitimate input for~ISS
must have finite variation.


\emph{Statement of contributions:} The contributions of this paper are
twofold.  Our first contribution concerns the noise-to-state stability
of systems described by SDEs with persistent noise.  We
generalize the notion of noise-dissipative Lyapunov function, which is
a positive semidefinite function that satisfies a dissipation
inequality that can be nonexponential (by this we mean that the
inequality admits a convex~$\classkinfty$ gain instead of the linear
gain characteristic of exponential dissipativity).  We also introduce
the concept of $p$thNSS-Lyapunov function with respect to a closed
set, which is a noise-dissipative Lyapunov function that in addition
is proper with respect to the set with a convex lower-bound gain
function.  Using this framework, we show that noise-dissipative
Lyapunov functions have NSS~dynamics and we characterize the overshoot
gain.  More importantly, we show that the existence of a
\textit{p}thNSS-Lyapunov function with respect to a closed set implies
that the system is NSS in \textit{p}th moment with respect to the set.
Our second contribution is driven by the aim of providing alternative,
structured ways to check the hypotheses of the above results.  We
introduce the notion of two functions being proper with respect to
each other as a generalization of the notion of properness with
respect to a set.  We then develop a methodology to verify whether two
functions are proper with respect to each other by analyzing the
associated pair of inequalities with increasingly strong refinements
that involve the classes~$\classk$, $\classkinfty$, and $\classkinfty$
plus a convexity property.  We show that these refinements define
equivalence relations between pairs of functions, thereby producing
nested partitions on the space of functions.  This provides a useful
way to deal with these inequalities because the construction of the
gains is explicit when the transitivity property is exploited. This
formalism motivates our characterization of positive semidefinite
functions that are proper, in the various refinements, with respect to
the Euclidean distance to their nullset.  This characterization is
technically challenging because we allow the set to be noncompact, and
thus the pre-comparison functions can be discontinuous. We devote
special attention to the case when the set is a subspace and examine
the connection with seminorms.  Finally, we show how this framework
allows us to develop an alternative formulation of our stability
results.


\emph{Organization:} The paper is organized as
follows. Section~\ref{sec:preliminaries} introduces preliminaries on
seminorms, comparison functions, and SDEs.
Section~\ref{sec:problem-statement} presents the NSS stability results
and Section~\ref{sec:positive semidefinite-functions} develops the
methodology to help verify their hypotheses. Finally,
Section~\ref{sec:conclusions-future} discusses our conclusions and
ideas for future work.

\section{Preliminary notions}\label{sec:preliminaries}

This section reviews some notions on comparison functions and
stochastic differential equations that are used throughout the paper.

\subsection{Notational conventions}\label{sec:notation}

Let $\real$ and $\realnonnegative$ be the sets of real and nonnegative
real numbers, respectively.  We denote by~$\real^n$ the
$n$-dimensional Euclidean space.  A subspace $\Uset\subseteq\real^n$
is a subset which is also a vector space. Given a matrix
$A\in\realmatricesrectangulararg{m}{n}$, its nullspace
$\kernel(A)\triangleq\{x\in\real^n:Ax=0\}$ is a subspace.  Given
$\dom\subseteq\real^n$, we denote by~$\psddomain$
and~$\psddomaintwice$ the set of positive semidefinite functions
defined on $\dom$ that are continuous and continuously twice
differentiable (if $\dom$ is open), respectively. Given
$\lyap\in\psdtwice$, we denote its gradient by~$\grad\lyap$ and its
Hessian by~$\hessian\lyap$.  A seminorm is a function
$S:\real^n\to\real$ that is positively homogeneous, i.e., $S(\lambda
x) = \absolute{\lambda} S(x)$ for any $\lambda\in \real$, and
satisfies the triangular inequality, i.e., $S(x+y)\le S(x)+S(y)$ for
any $x,y\in \real^n$. From these properties it can be deduced that
$S\in\psd$ and its nullset is always a subspace. If, moreover, the
function $S$ is positive definite, i.e., $S(x)=0$ implies $x=0$, then
$S$ is a norm. The Euclidean norm of $x\in\real^n$ is denoted
by~$\norm{x}$, and the Frobenius norm of the matrix
$A\in\realmatricesrectangulararg{m}{n}$ is
$\normfrob{A}\triangleq\sqrt{\trace{(A\tp A)}} = \sqrt{\trace{(A
    A\tp)}}$. For any matrix $A\in\realmatricesrectangulararg{m}{n}$,
the function $\seminorm{x}{A}\triangleq\norm{Ax}$ is a seminorm and
can be viewed as a distance to $\kernel(A)$. For a symmetric positive
semidefinite real matrix $A\in\realmatrices$, we order its eigenvalues
as $\lambdamax(A)\triangleq\lambda_1(A)\ge
\dots\ge\lambda_n(A)\triangleq\lambdamin(A)$, so if the dimension of
$\kernel(A)$ verifies $\dim(\kernel(A))=k\le n$, then
$\lambda_{n-k}(A)$ is the minimum nonzero eigenvalue of~$A$.
The Euclidean distance from $x$ to a set $\Uset \subseteq\real^n $ is
defined by
$\distU{x}\triangleq\inf\setdef{\norm{x-u}}{u\in\Uset}$. The function
$\distU{.}$ is continuous when~$\Uset$ is closed.  Given
$\map{f,g}{\realnonnegative}{\realnonnegative}$, we say that $f(s)$ is
in $\Obig(g(s))$ as $s\to\infty$ if there exist constants $\kappa,
s_0>0$ such that $f(s)<\kappa g(s)$ for all $s>s_0$.

\subsection{Comparison, convex, and concave
  functions}\label{sec:pre-classks-convexity}

Here we introduce some classes of comparison functions
following~\cite{HKK:02} that are useful in our technical treatment. A
continuous function $\alpha:[0,b)\to\realnonnegative$, for $b>0$ or
$b=\infty$, is class $\classk$ if $\alpha(0)=0$ and is strictly
increasing. A function
$\map{\alpha}{\realnonnegative}{\realnonnegative}$ is
class~$\classkinfty$ if $\alpha\in\classk$ and is unbounded. A
continuous function
$\map{\mu}{\realnonnegative\times\realnonnegative}{\realnonnegative}$
is class~$\classkl$ if, for each fixed $s\ge 0$, the function $r
\mapsto \mu(r,s)$ is class $\classk$, and, for each fixed $r\ge 0$,
the function $s\mapsto \mu(r,s)$ is decreasing and
$\lim_{s\to\infty}\mu(r,s)=0$.  If $\alpha_1$, $\alpha_2$ are
class~$\classk$ and the domain of $\alpha_1$ contains the range of
$\alpha_2$, then their composition $\alpha_1\circ\alpha_2$ is
class~$\classk$ too.  If $\alpha_3$, $\alpha_4$ are
class~$\classkinfty$, then both the inverse function $\alpha_3^{-1}$
and their composition $\alpha_3\circ\alpha_4$ are
class~$\classkinfty$. In our technical treatment, it is sometimes
convenient to require comparison functions to satisfy additional
convexity properties.  A real-valued function $f$ defined in a convex
set $X$ in a vector space is convex if $f(\lambda x+(1-\lambda)y)\le
\lambda f(x) + (1-\lambda)f(y)$ for each $x,y\in X$ and any
$\lambda\in[0,1]$, and is concave if $-f$ is
convex. By~\cite[Ex.~3.3]{SB-LV:09}, if $f:[a,b]\to[f(a),f(b)]$ is a
strictly increasing convex (respectively, concave) function, then the
inverse function $f^{-1}:[f(a),f(b)]\to [a,b]$ is strictly increasing
and concave (respectively, convex). Also, following~\cite[Section
3]{SB-LV:09}, if $f,g:\real\to\real$ are convex (respectively,
concave) and $f$ is nondecreasing, then the composition $f\circ g$ is
also convex (respectively, concave).

\subsection{Brownian motion}\label{sec:brownian}

We review some basic facts on probability and introduce the notion of
Brownian motion following~\cite{XM:11}.  Throughout the paper, we
assume that $(\Omega, \sigmaalgebra, \filtration, \probability)$ is a
complete probability space, where $\probability$ is a probability
measure defined on 
the $\sigma$-algebra $\sigmaalgebra$, which contains all the subsets
of $\Omega$ of probability $0$. The filtration $\filtration$ is a
family of sub-$\sigma$-algebras of $\sigmaalgebra$ satisfying
$\filtrationt\subseteq\filtrations\subseteq\sigmaalgebra$ for any
$0\le t < s < \infty$; we assume it is right continuous, i.e.,
$\filtrationt=\cap_{s>t}\filtrations$ for any $t\ge0$, and
$\mathcal{F}_0$ contains all the subsets of $\Omega$ of
probability~$0$.  The Borel $\sigma$-algebra in $\real^n$, denoted by
$\borel^n$, or in $[t_0,\infty)$, denoted by $\borel([t_0,\infty))$,
are the smallest $\sigma$-algebras that contain all the open sets in
$\real^n$ or $[t_0,\infty)$, respectively. A function
$X:\Omega\to\real^n$ is $\sigmaalgebra$-measurable if the set
$\{\omega\in\Omega: X(\omega)\in A\}$ belongs to $\sigmaalgebra$ for
any $A\in\borel^n$. We call such function a
($\sigmaalgebra$-measurable) $\real^n$-valued random variable. If $X$
is a real-valued random variable that is integrable with respect to
$\probability$, its expectation is $\expectation [X]=\int_{\Omega}
X(\omega) \differential\probability(\omega)$. A function
$f:\Omega\times[t_0,\infty)\to\real^n$ is
$\sigmaalgebra\times\borel$-measurable (or just measurable) if the set
$\{(\omega, t)\in\Omega\times[t_0,\infty): f(\omega, t)\in A\}$
belongs to $\sigmaalgebra\times\borel([t_0,\infty))$ for any
$A\in\borel^n$. We call such function an $\filtrationadapted$-adapted
process if $\map{f(.,t)}{\Omega}{\real^n}$ is
$\filtrationt$-measurable for every $t\ge t_0$. At times, we omit the
dependence on ``$\omega$'', in the sense that we refer to the indexed
family of random variables, and refer to the random process
$f=\fprocess$. We define $\lone$ as the set of all $\real^n$-valued
measurable $\filtrationadapted$-adapted processes $f$ such that
$\probability (\setdef{\omega\in\Omega}{\int_{t_0}^T \norm{f(\omega,
    s)}\, \ds<\infty})=1$ for every~$T>t_0$. Similarly,
$\ltwomatrices$ denotes the set of all $\realnm$-matrix-valued
measurable $\filtrationadapted$-adapted processes $\galone$ such that
$\probability (\setdef{\omega\in\Omega} {\int_{t_0}^T
  \normfrob{\galone(\omega, s)}^2\, \ds<\infty})=1$ for every~$T>t_0$.

A one-dimensional Brownian motion
$\map{\brownian}{\Omega\times[t_0,\infty)}{\real}$ defined in the
probability space $(\Omega, \sigmaalgebra, \filtration, \probability)$
is an $\filtrationadapted$-adapted process such that
\begin{itemize}
\item 
$\probability (\setdef{\omega\in\Omega}{\brownian(\omega,
    t_0)=0})=1$;

\item the mapping $\map{\brownian(\omega, .)}{[t_0,\infty)}{\real}$,
  called sample path, is continuous also with probability~$1$;

\item the increment $\brownian(.,t)-\brownian(.,s):\Omega\to\real$ is
  independent of $\filtrations$ for $t_0 \le s < t < \infty$ (i.e., if
  $S_b\triangleq\{\omega\in\Omega: \brownian(\omega,
  t)-\brownian(\omega, s)\in(-\infty, b)\}$, for $b\in\real$, then
  $\probability(A\cap S_b)= \probability(A)\probability(S_b)$ for all
  $A\in\filtrations$ and all $b \in \real$). In addition, this
  increment is normally distributed with zero mean and variance $t-s$.
\end{itemize}

An $m$-dimensional Brownian motion
$\map{\brownian}{\Omega\times[t_0,\infty)}{\real^m}$ is given by
$\brownian(\omega, t)= [\brownian_1(\omega,
t),\dots,\brownian_m(\omega, t)]\tp$, where each $\brownian_i$ is a
one-dimensional Brownian motion and, for each $t\ge t_0$, the random variables
$\brownian_1(t),...,\brownian_m(t)$ are independent.

\subsection{Stochastic differential equations}\label{sec:prel-sdes}
Here we review some basic notions on stochastic differential equations
(SDEs) following~\cite{XM:11}; other useful references
are~\cite{RK:12,BO:10,JRM:11}.  Consider the $n$-dimensional SDE
\begin{align}\label{eq:nonlinear-SDE-preliminaries}
  \differential x(\omega, t)= f\big(x(\omega, t),t\big)\dt
  +G\big(x(\omega, t),t\big)\Sigma(t)\,\db(\omega, t),
\end{align}
where $x(\omega, t)\in\real^n$ is a realization at time $t$ of the
random variable $\map{x(.,t)}{\Omega}{\real^n}$, for
$t\in[t_0,\infty)$. 
The initial condition is given by $x(\omega, t_0)=x_0$ with
probability~$1$ for some $x_0\in\real^n$.  The functions
$\map{f}{\real^n\times[t_0,\infty)}{\real^n}$,
$\galone:\real^n\times[t_0,\infty)\to\realmatricesrectangulararg{n}{q}$,
and $\map{\Sigma}{[t_0,\infty)}{\realmatricesrectangulararg{q}{m}}$
are measurable.  The functions $f$ and $\galone$ are regarded as a
model for the architecture of the system and, instead, $\Sigma$ is
part of the model for the stochastic disturbance; at any given time
$\Sigma$ determines a linear transformation of the $m$-dimensional
Brownian motion $\{\brownian(t)\}_{t\ge t_0}$, so that at time $t\ge
t_0$ the input to the system is
the process $\{\Sigma(t)\brownian(t)\}_{t\ge t_0}$, with covariance
$\int_{t_0}^t \Sigma(t)\Sigma(t)\tp \ds$.  The distinction between the
roles of~$\galone$ and~$\Sigma$ is irrelevant for the SDE; both
together determine the effect of the Brownian motion. 
The integral form of~\eqref{eq:nonlinear-SDE-preliminaries} is
given by
\begin{align*}
  x(\omega, t)=x_0+\int_{t_0}^t f\big(x(\omega, s),s\big) \ds +
  \int_{t_0}^t G\big(x(\omega, s),s\big)\Sigma(s)\,\db(\omega, s),
\end{align*}
where the second integral is an stochastic
integral~\cite[p. 18]{XM:11}.
A $\real^n$-valued random process $\xprocess$ is a solution
of~\eqref{eq:nonlinear-SDE-preliminaries} with initial value $x_0$ if
\begin{enumerate}
\item is continuous with probability~$1$,
  $\filtrationadapted$-adapted, and satisfies $x(\omega, t_0)=x_0$
  with probability~$1$,
\item the processes $\fprocesssolution$ and $\gprocesssolution$ belong
  to $\lone$ and $\ltwomatrices$ respectively, and
\item equation~\eqref{eq:nonlinear-SDE-preliminaries} holds for every
  $t\ge t_0$ with probability~$1$.
\end{enumerate}
A solution $\xprocess$ of~\eqref{eq:nonlinear-SDE-preliminaries} is
unique if any other solution $\xbarprocess$ with $\bar{x}(t_0)=x_0$
differs from it only in a set of probability~$0$, that is,
$\probability (\big\{x(t)=\bar{x}(t)$\; $\forall\, t\ge t_0\big\})=1$.

We make the following assumptions on the objects
defining~\eqref{eq:nonlinear-SDE-preliminaries} to guarantee existence
and uniqueness of solutions.

\begin{assumption}\label{ass:assumptions-SDE}
  We assume $\Sigma$ is essentially locally bounded. Furthermore, for
  any $T>t_0$ and $n\ge 1$, we assume there exists $K_{T,n}>0$ such
  that, for almost every $t\in[t_0,T]$ and all $x,y\in\real^n$ with
  $\max\big\{\norm{x} ,\norm{y}\big\}\le n$,
  \begin{align*}
    \max&\big\{\,\norm{f(x,t)-f(y,t)}^2\,,\,\normfrob{\gtx-\gty}^2\,\big\}
    \le K_{T,n}\norm{x-y}^2.
  \end{align*}
  Finally, we assume that for any $T>t_0$, there exists $K_T>0$ such
  that, for almost every $t\in[t_0,T]$ and all $x\in\real^n$, $
  x\tp f(x,t)+\tfrac{1}{2}\normfrob{\gtx}^2\le K_T(1+\norm{x}^2)$.
\end{assumption}

According to~\cite[Th. 3.6, p. 58]{XM:11},
Assumption~\ref{ass:assumptions-SDE} is sufficient to guarantee global
existence and uniqueness of solutions
of~\eqref{eq:nonlinear-SDE-preliminaries} for each initial condition
$x_0\in\real^n$.

We conclude this section by presenting a useful operator in the
stability analysis of SDEs.  Given a function $\lyap\in\psdtwice$, we
define the generator of~\eqref{eq:nonlinear-SDE-preliminaries} acting
on the function $\lyap$ as the mapping
$\map{\itowoarguments{\lyap}}{\real^n\times[t_0,\infty)}{\real}$ given
by
\begin{align}\label{eq:Ito-operator}
  \ito{\lyap}{x,t}\triangleq\grad\lyap(x)\tp f(x,t)
  +\tfrac{1}{2}\trace\Big(\Sigma(t)\tp\gtx\tp
  \hessian\lyap(x)\gtx\Sigma(t)\Big).
\end{align}
It can be shown that $\ito{\lyap}{x,t}$ gives the expected rate of
change of $\lyap$ along a solution
of~\eqref{eq:nonlinear-SDE-preliminaries} that passes through the
point~$x$ at time~$t$, so it is a generalization of the Lie
derivative. According to~\cite[Th. 6.4, p. 36]{XM:11}, if we evaluate $\lyap$
along the solution $\xprocess$
of~\eqref{eq:nonlinear-SDE-preliminaries}, then the process
$\{\lyap(x(t))\}_{t\ge t_0}$ satisfies the new SDE
\begin{align}\label{eq:sde-lyapunov-function}
  \lyap(x(t)) & = \lyap(x_0) +\int_{t_0}^t\ito{\lyap}{x(s),s}\ds +
  \int_{t_0}^t\grad\lyap(x(s))\tp\gs\Sigma(s)\db(s).
\end{align}
Equation~\eqref{eq:sde-lyapunov-function} is known as It\^o's formula
and corresponds to the stochastic version of the chain rule.

\section{Noise-to-state stability via noise-dissipative Lyapunov
  functions}\label{sec:problem-statement}

In this section, we study the stability of stochastic differential
equations subject to persistent noise. Our first step is the
introduction of a novel notion of stability.  This captures the
behavior of the $p$th moment of the distance (of the state) to a
given closed set, as a function of two objects: the initial condition
and the maximum size of the covariance. After this, our next step is
to derive several Lyapunov-type stability results that help determine
whether a stochastic differential equation enjoys these stability
properties.
The following definition generalizes the concept of noise-to-state
stability given in~\cite{HD-MK:00}.

\begin{definition}\longthmtitle{Noise-to-state stability with respect
    to a set}\label{def:NSS-expectation-subspace}
  The system~\eqref{eq:nonlinear-SDE-preliminaries} is
  \emph{noise-to-state stable (NSS) in probability with respect to the
    set}~$\Uset\subseteq\real^n$ if for any $\epsilon>0$ there exist
  $\mu\in\classkl$ and~$\theta\in\classk$ (that might depend
  on~$\eps$), such that
  \begin{align}\label{eq:def-ISS-probability}
    \probability \Big\{\,\distU{x(t)}^p > \mu\big(\distU{x_0},
    t-t_0\big) + \theta\Big(\esssup_{t_0\le s\le
      t}\normfrob{\Sigma(s)}\Big)\,\Big\} \le\eps,
  \end{align}
  for all $t\ge t_0$ and any $x_0\in\real^n$.  And the
  system~\eqref{eq:nonlinear-SDE-preliminaries} is \emph{$p$th moment
    noise-to-state stable ($p$thNSS) with respect to}~$\Uset$ if there
  exist $\mu\in\classkl$ and $\theta\in\classk$, such that
  \begin{align}\label{eq:def-ISS-expectation}
    \expectationarg{\distU{x(t)}^p} & \le
    \mu\big(\distU{x_0}, t-t_0\big) + \theta\Big(\esssup_{t_0\le
      s\le t}\normfrob{\Sigma(s)}\Big) ,
  \end{align}
  for all $t\ge t_0$ and any $x_0\in\real^n$.  The gain functions
  $\mu$ and $\theta$ are the \emph{overshoot gain} and the
  \emph{noise gain}, respectively.
\end{definition}

The quantity $\normfrob{\Sigma(t)} = \sqrt{ \trace\big(\Sigma(t)
  \Sigma(t)\tp\big)}$ is a measure of the size of the noise because it
is related to the infinitesimal covariance
$\Sigma(t)\Sigma(t)\tp$. The choice of the $p$th power is irrelevant
in the statement in probability since one could take any
$\classkinfty$ function evaluated at $\distU{x(t)}$. However, this
would make a difference in the statement in expectation. (Also, we use the
same power for convenience.) When the set $\Uset$ is a subspace, we can
substitute~$\distU{.}$ by $\seminorm{.}{A}$, for some matrix
$A\in\realmatricesrectangulararg{m}{n}$ with $\kernel(A)=\Uset$. In
such a case, the definition above does not depend on the choice of the
matrix~$A$.

\begin{remark}\longthmtitle{NSS is not a particular case of ISS}
  {\rm The concept of NSS is not a particular case of input-to-state
    stability (ISS)~\cite{EDS:08} for systems that are affine in the
    input, namely,
    \begin{align*} 
      \dot{y} = f(y,t) + \galone(y,t) u(t) \;\:\Leftrightarrow\;\:
      y(t)=y(t_0)+ \int_{t_0}^t f(y(s),s)\,\ds +\int_{t_0}^t G(y(s),s)
      u(s)\,\ds,
    \end{align*}
    where $\map{u}{[t_0, \infty)}{\real^q}$ is measurable and
    essentially locally bounded~\cite[Sec. C.2]{EDS:98}. The reason is
    the following: the components of the vector-valued function
    $\int_{t_0}^t G(y(s),s) u(s)\,\ds$ are differentiable almost
    everywhere by the Lebesgue fundamental theorem of
    calculus~\cite[p. 289]{JNM-NAW:99}, and thus absolutely
    continuous~\cite[p. 292]{JNM-NAW:99} and with bounded
    variation~\cite[Prop. 8.5]{JNM-NAW:99}. On the other hand, at
    any time previous to $t_k(t)\triangleq\min\{t,\inf{\setdef{s\ge t_0}{\norm{x(s)}\ge k}}\}$, 
    the
        driving disturbance of~\eqref{eq:nonlinear-SDE-preliminaries} is
    the vector-valued function $\int_{t_0}^{t_k(t)}
    \gs\Sigma(s)\db(s)$, whose $i$th component has quadratic
    variation~\cite[Th. 5.14, p. 25]{XM:11} equal to
    \begin{align*}
      \int_{t_0}^{t_k(t)}\sum\limits_{j=1}^{m}
      \absolute{\sum\limits_{l=1}^{q} \gs_{il}\Sigma(s)_{lj}}^2\ds > 0.
    \end{align*}
    Since a continuous process that has positive quadratic variation
    must have infinite variation~\cite[Th. 1.10]{FCK:05}, we conclude
    that the driving disturbance in this case is not allowed in the
    ISS framework.}  \oprocend
\end{remark}

Our first goal now is to provide tools to establish whether a
stochastic differential equation enjoys the noise-to-state stability
properties given in Definition~\ref{def:NSS-expectation-subspace}. To
achieve this, we look at the dissipativity properties of a special
kind of energy functions along the solutions
of~\eqref{eq:nonlinear-SDE-preliminaries}.

\begin{definition}\longthmtitle{Noise-dissipative Lyapunov
    function}\label{def:noise-dissipative-Lyapunov}
  A function $\lyap \in \psdtwice$ is a \emph{noise-dissipative
    Lyapunov function} for~\eqref{eq:nonlinear-SDE-preliminaries} if
  there exist $\lyapw\in\psd$, $\gainnoise\in\classk$, and concave
  $\increasingw\in\classkinfty$ such that
  \begin{align}\label{eq:theorem-second-hypothesis-VandW}
    \lyap(x)\le\increasingw(\lyapw(x)),
  \end{align}
  for all $x\in\real^n$, and the following dissipation inequality holds:
  \begin{align}\label{eq:theorem-hypothesis-ito}
    \itowoarguments{\lyap}(x,t)\le-\lyapw(x) +
    \gainnoise\big(\normfrob{\Sigma(t)}\big),
  \end{align}
  for all $(x,t)\in\real^n\times[t_0,\infty)$.
\end{definition}

\begin{remark}\longthmtitle{It\^o formula and exponential
    dissipativity}\label{re:convex-dissipativity-vs-exponential-dissipativity}
  {\rm Interestingly, the
    conditions~\eqref{eq:theorem-second-hypothesis-VandW}
    and~\eqref{eq:theorem-hypothesis-ito} are equivalent to
    \begin{align}\label{eq:theorem-dissipative-condition}
      \itowoarguments{\lyap}(x,t)\le-\increasingw^{-1}(\lyap(x))+
      \gainnoise\big(\normfrob{\Sigma(t)}\big),
    \end{align}
    for all $x\in\real^n$, where $\increasingw^{-1}\in\classkinfty$ is
    convex.  Note that, since $\itowoarguments{\lyap}$ is not the Lie
    derivative of $\lyap$ (as it contains the Hessian of $\lyap$), one
    cannot directly deduce
    from~\eqref{eq:theorem-dissipative-condition} the existence of a
    continuously twice differentiable function $\tilde{\lyap}$ such that
    \begin{align}\label{eq:exponential-dissipativity}
      \itowoarguments{\tilde{\lyap}}(x,t)\le-c\tilde{\lyap}(x)+
      \tilde{\gainnoise}\big(\normfrob{\Sigma(t)}\big),
    \end{align}
    as instead can be done in the context of~ISS, see
    e.g.~\cite{LP-YW:96}.}  \oprocend
\end{remark}
  
\begin{example}\longthmtitle{A noise-dissipative Lyapunov
    function}\label{ex:example}
  \rm{Assume that $\map{\convexf}{\real^n}{\real}$ is continuously
    differentiable and verifies
    \begin{align}\label{eq:convex-uniformly-convex}
      \gamma(\norm{x-x'}^2)\le(x-x')\tp(\grad
      \convexf(x)-\grad\convexf(x'))
    \end{align}
    for some convex function~$\gamma\in\classkinfty$ for all $x,
    x'\in\real^n$. In particular, this implies that~$h$ is
    strictly convex. (Incidentally, any strongly convex function
    verifies~\eqref{eq:convex-uniformly-convex} for some choice
    of~$\gamma$ linear and strictly increasing.) Consider now the
    dynamics
    \begin{align}\label{eq:inexact-distributed-opt}
      \differential x(\omega, t)=-\big(\delta\lap x(\omega, t)
      +\grad\convexf(x(\omega, t))\big)\dt +\Sigma(t)\,\db(\omega, t),
    \end{align}
    for all $t\in[t_0,\infty)$, where $x(\omega, t_0)=x_0$ with
    probability~$1$ for some $x_0\in\real^n$, and $\delta>0$.  Here,
    the matrix $\lap\in\realmatrices$ is symmetric and positive
    semidefinite, and the matrix-valued function
    $\map{\Sigma}{[t_0,\infty)}{\realmatricesrectangulararg{n}{m}}$ is
    continuous.
    This dynamics
    corresponds to the SDE~\eqref{eq:nonlinear-SDE-preliminaries} with
    $f(x,t)\triangleq-\delta\lap x -\grad\convexf(x)$ and
    $\galone(x,t)\triangleq\identity_n$ for all
    $(x,t)\in\real^n\times[t_0,\infty)$.
  
    Let $x^*\in\real^n$ be the unique solution of the
    Karush-Kuhn-Tucker~\cite{SB-LV:09} condition $\delta\lap
    x^*=-\grad\convexf(x^*)$, corresponding to the unconstrained
    minimization of $F(x)\triangleq\tfrac{\delta}{2}x\tp\lap
    x+\convexf(x)$.  Consider then the candidate Lyapunov function
    $\lyap\in\psdtwice$ given by $\lyap(x)\triangleq
    \tfrac{1}{2}(x-x^*)\tp(x-x^*)$.  Using~\eqref{eq:Ito-operator}, we
    obtain that, for all $x\in\real^n$,
    \begin{align*} 
      \ito{\lyap}{x,t} & = -(x-x^*)\tp\Big(\delta\lap x
      +\grad\convexf(x)\Big)
      +\tfrac{1}{2}\trace\Big(\Sigma(t)\tp\Sigma(t)\Big)\nonumber
      \\
      &=- \delta(x-x^*)\tp\lap
      (x-x^*)-(x-x^*)\tp\Big(\grad\convexf(x)-\grad\convexf(x^*)\Big)
      +\tfrac{1}{2}\normfrob{\Sigma(t)}^2\nonumber
      \\
      &\le -\gamma(\norm{x-x^*}^2)+\tfrac{1}{2}\normfrob{\Sigma(t)}^2.
    \end{align*}
    We note that $\lyapw\in\psd$ defined by $\lyapw(x)\triangleq
    \gamma(\norm{x-x^*}^2)$ verifies
    \begin{align*} 
      \lyap(x)=\tfrac{1}{2}\gamma^{-1}\big(\lyapw(x)\big)\;\;\quad\forall
      x\in\real^n,
    \end{align*}
    where~$\gamma^{-1}$ is concave and belongs to the
    class~$\classkinfty$ as explained in
    Section~\ref{sec:pre-classks-convexity}.  Therefore, $\lyap$ is a
    noise-dissipative Lyapunov function
    for~\eqref{eq:inexact-distributed-opt}, with concave
    $\increasingw\in\classkinfty$ given by
    $\increasingw(r)=1/2\gamma^{-1}(r)$ and $\gainnoise\in\classk$
    given by $\gainnoise(r)\triangleq 1/2\,r^2$. \oprocend}
\end{example}

The next result generalizes~\cite[Th. 4.1]{HD-MK-RJW:01} to positive
semidefinite Lyapunov functions that satisfy weaker dissipativity
properties (cf.~\eqref{eq:theorem-dissipative-condition}) than the
typical exponential-like
inequality~\eqref{eq:exponential-dissipativity}, and characterizes the
overshoot gain.

\begin{theorem}\longthmtitle{Noise-dissipative Lyapunov functions have
    an NSS dynamics}\label{th:Stability-Non-Linear-Systems}
  Under Assumption~\ref{ass:assumptions-SDE}, 
  and further assuming that~$\Sigma$ is continuous,
   suppose that~$\lyap$ is
  a noise-dissipative Lyapunov function
  for~\eqref{eq:nonlinear-SDE-preliminaries}.  Then,
  \begin{align}\label{eq:ISS-theorem}
    \expectationarg{\lyap(x(t))} &
    \:\le\muISStilde\big(\lyap(x_0),t-t_0\big)+\increasingw
    \Big(2\,\gainnoise\big(\max_{t_0\le s\le
      t}\normfrob{\Sigma(s)}\big)\Big),
  \end{align}
  for all $t\ge t_0$, where the class~$\classkl$ function $(r,s)
  \mapsto \muISStilde(r,s)$ is well defined as the solution $y(s)$ to the
  initial value problem
  \begin{align}\label{eq:class-kl-beta-satisfies}
    \dot{y}(s) = -\tfrac{1}{2}{\increasingw}^{-1}(y(s)),\quad y(0)=r.
  \end{align}
\end{theorem}
\begin{proof}
  Recall that Assumption~\ref{ass:assumptions-SDE} guarantees the
  global existence and uniqueness of solutions
  of~\eqref{eq:nonlinear-SDE-preliminaries}. Given the process
  $\{\lyap(x(t))\}_{t\ge t_0}$, the proof strategy is to obtain a differential inequality 
  for $\expectationarg{\lyap(x(t))}$ using It\^o formula~\eqref{eq:sde-lyapunov-function}, 
  and then use a
  comparison principle to translate the problem into one of standard
  input-to-state stability for an appropriate choice of the input. 
  %
  %
  
    To carry out this strategy, we consider It\^o formula~\eqref{eq:sde-lyapunov-function} with respect to an arbitrary reference time instant $t'\ge t_0$,
    \begin{align}\label{eq:sde-lyapunov-function-arbitrary-reference}
      \lyap(x(t)) & = \lyap(x(t')) +\int_{t'}^t\ito{\lyap}{x(s),s}\ds +
      \int_{t'}^t\grad\lyap(x(s))\tp\gs\Sigma(s)\db(s),
    \end{align}
    and we first ensure that the expectation
    of the integral against Brownian motion is~$0$.
 Let $S_k=\setdef{x\in\real^n}{\norm{x}\le k}$ be the ball of
  radius~$k$ centered at the origin. Fix $x_0\in\real^n$ and denote by
  $\tau_k$ the first exit time of~$x(t)$ from $S_k$ for integer
  values of $k$ greater than~$\norm{x(t_0)}$, namely,
  $\tau_k\triangleq\inf{\setdef{s\ge t_0}{\norm{x(s)}\ge k}}$,
  for $k>\lceil \norm{x(t_0)}\rceil$. 
    Since the event
    $\setdef{\omega\in\Omega}{\tau_k\le t}$ belongs to~$\filtrationt$ for each $t\ge t_0$, it 
    follows that
    $\tau_k$ is an $\filtrationadapted$-stopping time for each $t\ge t_0$.
      Now, for each~$k$
      fixed, if we consider the random variable
      $t_k(t)\triangleq\min\{t,\tau_k\}$  and
 define~$I(t',t)$ as the stochastic
    integral in~\eqref{eq:sde-lyapunov-function-arbitrary-reference} 
    for any fixed $t'\in [t_0,t_k(t)]$,
    then the process
    $I(t', t_k(t))$
   has zero expectation
  as we show next. The function $\map{X}{S_k\times[t',t]}{\real}$
  given by $X(x,s)\triangleq\grad\lyap(x)\tp \galone(x,s)\Sigma(s)$
  is essentially bounded (in its domain),
  and thus 
  $\expectationarg{\int_{t'}^t \indicator{[t', t_k(t)]}(s)\, X(x(s),t)^2 \,\ds}<\infty$, 
  where
  $\indicator{[t', t_k(t)]}(s)$ is the indicator function of the set
  $[t', t_k(t)]$. 
  Therefore, 
$\expectationarg{I(t', t_k(t))}=0$
  by~\cite[Th. 5.16, p. 26]{XM:11}. Define now
  $\lyapbar(t)\triangleq\expectationarg{\lyap(x(t))}$ and
  $\lyapwbar(t)\triangleq\expectationarg{\lyapw(x(t))}$ in
  $\Gamma(t_0)\triangleq\setdef{t\ge t_0}{\lyapbar(t)<\infty}$.
   By the above,
taking expectations in~\eqref{eq:sde-lyapunov-function-arbitrary-reference} 
  and
  using~\eqref{eq:theorem-hypothesis-ito}, we obtain that 
  \begin{align}\label{eq:differential-ineq-r}
    \lyapbar(t_k(t)) & = \lyapbar(t') + \expectation
    \Big[\,\int_{t'}^{t_k(t)}\ito{\lyap}{x(s),s}\ds\, \Big]\nonumber
    \\
    & \le \:\lyapbar(t') -
    \expectation\Big[\,\int_{t'}^{t_k(t)}\lyapw(x(s))\ds\, \Big] +
    \expectation\Big[\int_{t'}^{t_k(t)} \gainnoise(\normfrob{\Sigma(s)})\,\ds\,\Big]
  \end{align}
  for all $t\in\Gamma(t_0)$ and any $t'\in [t_0, t_k(t)]$. Next we use the fact that~$\lyap$ is
  continuous and $\{x(t)\}_{t\ge t_0}$ is also continuous with
  probability~$1$. In addition, according to Fatou's
  lemma~\cite[p. 123]{JNM-NAW:99} for convergence in the probability
  measure, we get that
  \begin{align}\label{eq:fatou-lyapbar}
    \lyapbar(t)=\,&\expectationarg{\,\lyap (x(\liminf_{k\to\infty}\,
      t_k(t)))\,} =\expectationarg{\,\liminf_{k\to\infty}\lyap (x(
      t_k(t)))\,}
    \\
    \le\,& \liminf_{k\to\infty}\,
    \expectationarg{\,\lyap(x(t_k(t)))\,}
    =\liminf_{k\to\infty}\lyapbar(t_k(t))\nonumber
  \end{align}
  for all $t\in\Gamma(t_0)$.
  Moreover, using the monotone convergence~\cite[p. 176]{JNM-NAW:99}
  when $k\to\infty$ in both Lebesgue integrals
  in~\eqref{eq:differential-ineq-r} (because both integrands are
  nonnegative and $\indicator{[t', t_k(t)]}$ converges monotonically
  to $\indicator{[t', t]}$ as $k\to\infty$ for any $t'\in [t_0, t_k(t)]$), we obtain
  from~\eqref{eq:fatou-lyapbar} that
  \begin{align}\label{eq:differential-ineq-beforeTonelli}
    \lyapbar(t)\le \lyapbar(t') -
    \expectation\Big[\,\int_{t'}^{t} \lyapw(x(s))\ds \,\Big]+
    \int_{t'}^{t}\gainnoise(\normfrob{\Sigma(s)})\,\ds
  \end{align}
   for all $t\in\Gamma(t_0)$ and any $t'\in[t_0,t]$.
  Before resuming the argument we make two observations. First,
  applying Tonelli's theorem~\cite[p. 212]{JNM-NAW:99} to the
  nonnegative process $\{\lyapw(x(s))\}_{s\ge t'}$, it follows that
  \begin{align}\label{eq:Tonelli_W}
  \expectation\big[\,\int_{t'}^{t} \lyapw(x(s))\ds\,\big]
    =\int_{t'}^{t} \lyapwbar(x(s))\ds.
  \end{align}
  Second, using~\eqref{eq:theorem-second-hypothesis-VandW} and
  Jensen's inequality~\cite[Ch. 3]{VSB:95}, 
  we get that 
  \begin{align}\label{eq:vbar-wbar-kinfcc}
    \lyapbar(t) = \expectationarg{\lyap(x(t))}
    \le\expectation\big[\increasingw(\lyapw(x(t)))\big]
    \le\increasingw\big(\expectationarg{\lyapw(x(t))}\big)
    =\increasingw\big(\lyapwbar(t)\big),
  \end{align}
  because~$\increasingw$ is concave, so
  $\lyapwbar(t)\ge{\increasingw}^{-1}(\lyapbar(t))$.
  Hence,~\eqref{eq:differential-ineq-beforeTonelli}
  and~\eqref{eq:Tonelli_W} yield
  \begin{align}\label{eq:integral-ineq-lyapbar}
    \lyapbar(t) &\:\le \lyapbar(t') -
    \int_{t'}^{t}\lyapwbar(s)\,\ds+
    \int_{t'}^{t}\gainnoise(\normfrob{\Sigma(s)})\,\ds 
    \nonumber
    \\&\:
    \le
    \lyapbar(t') +
    \int_{t'}^{t} \Big(-{\increasingw}^{-1}(\lyapbar(s))
    +\gainnoise(\normfrob{\Sigma(s)})\Big)\,\ds
  \end{align}
for all $t\in\Gamma(t_0)$ and any $t'\in [t_0, t]$,
 which in particular shows that $\Gamma(t_0)$ can be taken equal to $[t_0,\infty)$.
  
  Now the strategy is to compare~$\lyapbar$  with the unique solution of an ordinary differential 
  equation that represents an input-to-state stable (ISS) system. First we leverage the integral inequality~\eqref{eq:integral-ineq-lyapbar} to show that
  $\lyapbar$ is continuous in $[t_0,\infty)$, which allows us then to rewrite~\eqref{eq:integral-ineq-lyapbar} as a differential inequality at~$t'$. To 
  to show that~$\lyapbar$ is continuous, we use the dominated 
  convergence theorem~\cite[Thm. 2.3, P. 6]{XM:11} applied to
          $V_k(\hat{t})\triangleq\lyap(x(\hat{t}))-\lyap(x(\hat{t}+1/k))$, for $\hat{t}\in[t_0,t]$, and similarly taking $\hat{t}-1/k$  (excluding, respectively, the cases when $\hat{t}=t\,$ or $\,\hat{t}=t_0$). 
          The hypotheses are satisfied because $V_k$ can be majorized using~\eqref{eq:integral-ineq-lyapbar} as
         \begin{align}\label{eq:majorizing-function-lyapbar}
         \absolute{V_k(\hat{t})}
         \;\le\;\lyap(x(\hat{t}))+\lyap(x(\hat{t}+1/k))
         \; \le\;
                2\big( \lyap(x_0) +
                \int_{t_0}^{t} \gainnoise(\normfrob{\Sigma(s)})\,\ds\big),
         \end{align}
       where the term on the right is not a random variable and thus coincides with its expectation.
         Therefore, for every $\hat{t}\in[t_0, t]$,
          \begin{align*}
          \lim_{s\to \hat{t}}   
            \expectationarg{\lyap(x(s))}
            =\expectationarg{  \lim_{s\to \hat{t}}   \lyap(x(s))  }
            =\expectationarg{\lyap(x(\hat{t}))},
          \end{align*}
          so $\lyapbar$ is continuous on~$[t_0, t]$, for any $t\ge t_0$.
   Now, using again~\eqref{eq:integral-ineq-lyapbar} and the continuity of the integrand,
   we can bound the upper right-hand derivative~\cite[Appendix C.2]{HKK:02} 
   (also called upper Dini derivative), as
        \begin{align*}
        D^+\lyapbar(t')\triangleq &\, \limsup_{t\to t',\, t> t'}\frac{\lyapbar(t)-\lyapbar(t')}{t-t'}
        \\
        \le &\,
         \limsup_{t\to t',\, t> t'}\frac{1}{t-t'}\int_{t'}^t \big(-{\increasingw}^{-1}(\lyapbar(s))+\gainnoise(\normfrob{\Sigma(s)})\big)\,\ds
        \,=\,
        h(\lyapbar(t'),\Input(t')),
        \end{align*}
        for any $t'\in [t_0,\infty)$, where the function
                  $\map{h}{\realnonnegative\times\realnonnegative}{\real}$ is given by
                  \begin{align*}
                  h(y,\Input)\triangleq-{\increasingw}^{-1}(y)+\Input,
                  \end{align*}
                  and
        $\Input(t)\triangleq\gainnoise(\normfrob{\Sigma(t)})$, which is continuous in $[t_0,\infty)$.
    Therefore, according to the
          comparison principle~\cite[Lemma 3.4, P. 102]{HKK:02},
           using that $\lyapbar$ is continuous in $[t_0,\infty)$ and 
          $D^+\lyapbar(t')\le h(\lyapbar(t'),\Input(t'))$,  for any $t'\in [t_0,\infty)$,
          the solutions~\cite[Sec. C.2]{EDS:98} of the initial value problem
          \begin{align} \label{eq:comparison-principle-lemma-autoproper-implies-bounded-V} 
            \lyapudot(t) = h(\lyapu(t),\Input(t)),
            \qquad
            \lyapu_0\triangleq\lyapu(t_0)=\lyapbar(t_0)
          \end{align}
          (where~$h$ is locally Lipschitz in the first argument as we show next),
          satisfy that $\lyapu(t)\ge \lyapbar(t)\,(\ge 0)$ in the common interval of existence.         
   We argue 
   the global existence and uniqueness of 
  solutions
  of~\eqref{eq:comparison-principle-lemma-autoproper-implies-bounded-V}
  as follows.
  Since
  $\alpha\triangleq{\increasingw}^{-1}$ is convex and
  class~$\classkinfty$ (see Section~\ref{sec:pre-classks-convexity}),
  it holds that
  \begin{align*}
    \alpha(s')\le\alpha(s) \le
    \alpha(s')+\tfrac{\alpha(s'')-\alpha(s')}{s''-s'} (s-s')
  \end{align*}
  for all $s\in[s', s'']$, for any $s''>s'\ge 0$. Thus,
  $\absolute{\alpha(s)-\alpha(s')}=\alpha(s)-\alpha(s')\le L (s-s')$,
  for any $s''\ge s\ge s'\ge 0$, where $L\triangleq
  (\alpha(s'')-\alpha(s'))/(s''-s')$, so ${\increasingw}^{-1}$ is
  locally Lipschitz. Hence,~$h$ is locally Lipschitz in
  $\realnonnegative\times\realnonnegative$. 
  Therefore, given the
  input function~$\Input$ and any $\lyapu_0\ge 0$, there is a unique maximal
  solution
  of~\eqref{eq:comparison-principle-lemma-autoproper-implies-bounded-V}, denoted by
  $\lyapu(\lyapu_0, t_0; t)$, defined in a maximal interval $[t_0,
  t_{\text{max}}(\lyapu_0,t_0))$.  (As a by-product, the initial value
  problem~\eqref{eq:class-kl-beta-satisfies}, which can be written as
  $\dot{y}(s) = \tfrac{1}{2} h(y(s),0)$, $y(0)=r$, has a unique and
  strictly decreasing solution in~$[0,\infty)$, so~$\muISStilde$ in
  the statement is well defined and in class~$\classkl$.)
   To show
  that~\eqref{eq:comparison-principle-lemma-autoproper-implies-bounded-V}
  is ISS we follow a similar argument as in
  the proof of~\cite[Th. 5]{EDS:08} (and note that, as a consequence, we obtain that
  $t_{\text{max}}(\lyapu_0,t_0)=\infty$).  Firstly, if
  ${\increasingw}^{-1}(\lyapu) \ge 2\Input$, then $\lyapudot(t) =
  -\tfrac{1}{2}{\increasingw}^{-1}(\lyapu(t))$, which implies that~$\lyapu$ is nonincreasing outside the set
  $S\triangleq\setdef{t\ge t_0}{\lyapu(t)\le
    \increasingw(2\Input(t))}$. Thus, if some $t^*\ge t_0$ belongs to $S$, then so
  does every $t\in [t^*, t_{\text{max}}(\lyapu_0,t_0))$ implying that $\lyapu$ is 
  locally bounded because~$\Input$ is locally bounded (in fact, continuous). 
    (Note that $\lyapu(t)\ge 0$ because $\lyapudot(t)\ge 0$ whenever $\lyapu(t)=0$.)
  Therefore, for all $t\ge t_0$, and
   for~$\muISStilde$ as in the statement (which we have shown is well
   defined), we have that
   \begin{align*}
        \lyapbar(t)\,
               \le
     \lyapu(t)
     \le
     \max\Big\{\muISStilde\big(\lyapbar(t_0),t-t_0\big)\,,\,
     \increasingw\Big(2\max_{t_0\le s\le t}\Input(s)\Big)\Big\}.
   \end{align*}
   Since the maximum of two quantities is upper bounded by the sum, and
   using the definition of~$\Input$ together with the monotonicity
   of~$\gainnoise$, it follows that
   \begin{align}\label{eq:ISS-lyapu-to-make-comparison}
     \lyapbar(t)\,
            \le
      \lyapu(t)
      \le\muISStilde\big(\lyap(x_0),t-t_0\big) 
    + \increasingw\Big(2\gainnoise\Big(\max_{t_0\le s\le t}\normfrob{\Sigma(s)}\Big)\Big),
   \end{align}
   for all $t\ge t_0$,
   where we also used that~$\lyapbar(t_0)=\lyap(x_0)$, and the proof is complete. 
\end{proof}

Of particular interest to us is the case when the function $\lyap$ is
lower and upper bounded by class~$\classkinfty$ functions of the
distance to a closed, not necessarily bounded, set.

\begin{definition}\longthmtitle{NSS-Lyapunov
    functions}\label{def:Noise-to-state-Lyapunov}
  A function $\lyap \in \psdtwice$ is a \emph{strong NSS-Lyapunov
    function in probability with respect to}~$\Uset\subseteq\real^n$
  for~\eqref{eq:nonlinear-SDE-preliminaries} if $\lyap$ is a
  noise-dissipative Lyapunov function and, in addition, there exist
  $p>0$ and class~$\classkinfty$ functions $\gainNSSone$ and
  $\gainNSStwo$ such that
  \begin{align}\label{eq:theorem-third-hypothesis-V}
    \gainNSSone(\distU{x}^p) \le \lyap(x)\le\gainNSStwo(\distU{x}^p),
    \quad\forall x\in\real^n.
  \end{align}
  If, moreover, $\gainNSSone$ is convex, then $\lyap$ is a \emph{$p$th
    moment NSS-Lyapunov function with respect to}~$\Uset$. 
\end{definition}

Note that a strong NSS-Lyapunov function in probability with respect
to a set satisfies an inequality of the
type~\eqref{eq:theorem-third-hypothesis-V} for any $p>0$, whereas the
choice of~$p$ is relevant when $\gainNSSone$ is required to be convex.
The reason for the `strong' terminology is that we
require~\eqref{eq:theorem-dissipative-condition} to be satisfied with
convex $\increasingw^{-1}\in\classkinfty$. Instead, a standard
NSS-Lyapunov function in probability satisfies the same inequality
with a class~$\classkinfty$ function which is not necessarily convex.
We also note that~\eqref{eq:theorem-third-hypothesis-V} implies that
$\Uset=\setdef{x\in\real^n}{\lyap(x)=0}$, which is closed because
$\lyap$ is continuous.

\begin{example}\longthmtitle{Example~\ref{ex:example}--revisited:
   an NSS-Lyapunov function}\label{ex:example-two}
  {\rm Consider the function $\lyap$ introduced in
    Example~\ref{ex:example}.  For each $p\in(0,2]$, note that
    \begin{align*} 
      {\gainNSSone}_p(\norm{x-x^*}^p) \le
      \lyap(x)\le{\gainNSStwo}_p(\norm{x-x^*}^p)\;\;\quad\forall
      x\in\real^n,
    \end{align*}
    for the convex
    functions~${\gainNSSone}_p(r)={\gainNSStwo}_p(r)\triangleq
    r^{2/p}$, which are in the class~$\classkinfty$. (Recall that
    $\gainNSStwo$ in Definition~\ref{def:Noise-to-state-Lyapunov} is
    only required to be~$\classkinfty$.) Thus, the function~$\lyap$ is
    a $p$th moment NSS-Lyapunov function
    for~\eqref{eq:inexact-distributed-opt} with respect to~$x^*$ for
    $p\in(0,2]$. \oprocend}
\end{example}

The notion of NSS-Lyapunov function plays a key role in establishing
our main result on the stability of SDEs with persistent
noise.

\begin{corollary}\longthmtitle{The existence of an NSS-Lyapunov
    function implies the corresponding
    NSS property}\label{co:of-the-main-theorem} 
  Under Assumption~\ref{ass:assumptions-SDE}, 
    and further assuming that~$\Sigma$ is continuous,
  given a closed
  set~$\Uset \subset \real^n$,
  \begin{enumerate}
  \item if $\lyap \in \psdtwice$ is a strong NSS-Lyapunov function in
    probability with respect to~$\Uset$
    for~\eqref{eq:nonlinear-SDE-preliminaries}, then the system is NSS
    in probability with respect to~$\Uset$ with gain functions
    \begin{align}\label{eq:gains-NSS}
      \mu(r,s) \:\triangleq\gainNSSone^{-1}\big(\tfrac{2}{\eps}
      \muISStilde(\gainNSStwo(r^p),s ) \big),
      \quad\theta(r)\: \triangleq\gainNSSone^{-1}
      \big(\tfrac{2}{\eps}\increasingw(2\gainnoise(r))\big);
    \end{align}
    
  \item if $\lyap \in \psdtwice$ is a $p$thNSS-Lyapunov function with
    respect to $\Uset$ for~\eqref{eq:nonlinear-SDE-preliminaries},
    then the system is $p$th~moment NSS with respect to~$\Uset$ with
    gain functions $\mu$ and $\theta$ as in~\eqref{eq:gains-NSS}
    setting $\epsilon=1$.
  \end{enumerate}
\end{corollary}
\begin{proof}
  To show \emph{(i)}, note that, since $\gainNSSone(\distU{x}^p) \le
  \lyap(x)$ for all $x\in\real^n$, with $\gainNSSone\in\classkinfty$,
  it follows that for any $\hat{\rho}>0$ and $t\ge t_0$,
  \begin{align}\label{eq:corollary-bounded-probability}
    \probability\Big\{\distU{x(t)}^p > \hat{\rho}\Big\} & =
    \probability\Big\{\gainNSSone(\distU{x(t)}^p) >
    \gainNSSone(\hat{\rho}) \Big\}
    \le\probability\Big\{\lyap(x(t)) > \gainNSSone(\hat{\rho})\Big\}
    \le
    \frac{\expectationarg{\lyap(x(t))}}{\gainNSSone(\hat{\rho})}\nonumber
    \\
    & \le \frac{1}{\gainNSSone(\hat{\rho})}
    \bigg(\muISStilde\Big(\gainNSStwo(\distU{x_0}^p),\,t-t_0\Big)
    +\increasingw\Big(2\gainnoise \big(\max_{t_0\le s\le
      t}\normfrob{\Sigma(s)}\big)\Big)\bigg),
  \end{align}
  where we have used the strict monotonicity of $\gainNSSone$ in the
  first equation, Chebyshev's inequality~\cite[Ch. 3]{VSB:95} in the
  second inequality, and the upper bound for
  $\expectationarg{\lyap(x(t))}$ obtained in
  Theorem~\ref{th:Stability-Non-Linear-Systems},
  cf.~\eqref{eq:ISS-theorem}, in the last inequality (leveraging the
  monotonicity of $\muISStilde$ in the first argument and the fact
  that $\lyap(x)\le\gainNSStwo(\distU{x}^p)$ for all
  $x\in\real^n$). Also, for any function $\alpha\in\classk$, we have
  that $\alpha(2r)+\alpha(2s)\ge\alpha(r+s)$ for all $r,s\ge 0$. Thus,
  \begin{align}\label{eq:rhos-for-probability}
    \rho(\eps,x_0,t) & \triangleq\mu\big(\distU{x_0},t-t_0\big)
    +\theta\Big(\max_{t_0\le s\le t}\normfrob{\Sigma(s)}\Big)
    \\
    & \ge\gainNSSone^{-1}
    \Bigg(\frac{1}{\eps}\muISStilde\Big(\gainNSStwo(\distU{x_0}^p) ,
    t-t_0\Big) + \frac{1}{\eps}\increasingw\Big(2\gainnoise
    \big(\max_{t_0\le s\le t}\normfrob{\Sigma(s)}\big)\Big)\Bigg)
    \triangleq\hat{\rho}(\eps). \nonumber
  \end{align}
  Substituting now $\hat{\rho}\triangleq\hat{\rho}(\eps)$
  in~\eqref{eq:corollary-bounded-probability}, and
  using that $\rho(\eps,x_0,t)\ge \hat{\rho}(\eps)$,
  we get that
  $\probability\big\{\distU{x(t)}^p >
  \rho(\eps,x_0,t)\big\}\le\probability\big\{\distU{x(t)}^p >
  \hat{\rho}(\eps)\big\}\le\eps$.

  To show \emph{(ii)}, since $\gainNSSone^{-1}$ is concave, applying Jensen's
  inequality~\cite[Ch. 3]{VSB:95}, we get
  \begin{align*} 
    \expectationarg{\distU{x(t)}^p}
    \:\le\expectation\big[\gainNSSone^{-1}
    \big(\lyap(x(t))\big)\big]\le\gainNSSone^{-1}
    \Big(\expectationarg{\lyap(x(t))} \Big)
    \le\hat{\rho}(1)\le\rho(1,x_0,t),
  \end{align*}
  where in the last two inequalities we have used the bound for
  $\expectationarg{\lyap(x(t))}$
  in~\eqref{eq:corollary-bounded-probability} and the definition of
  $\hat{\rho}(\eps)$ in~\eqref{eq:rhos-for-probability}.
\end{proof}

\begin{figure*}[bth]
  \centering
  \subfigure[]{\includegraphics[width=.496\linewidth]{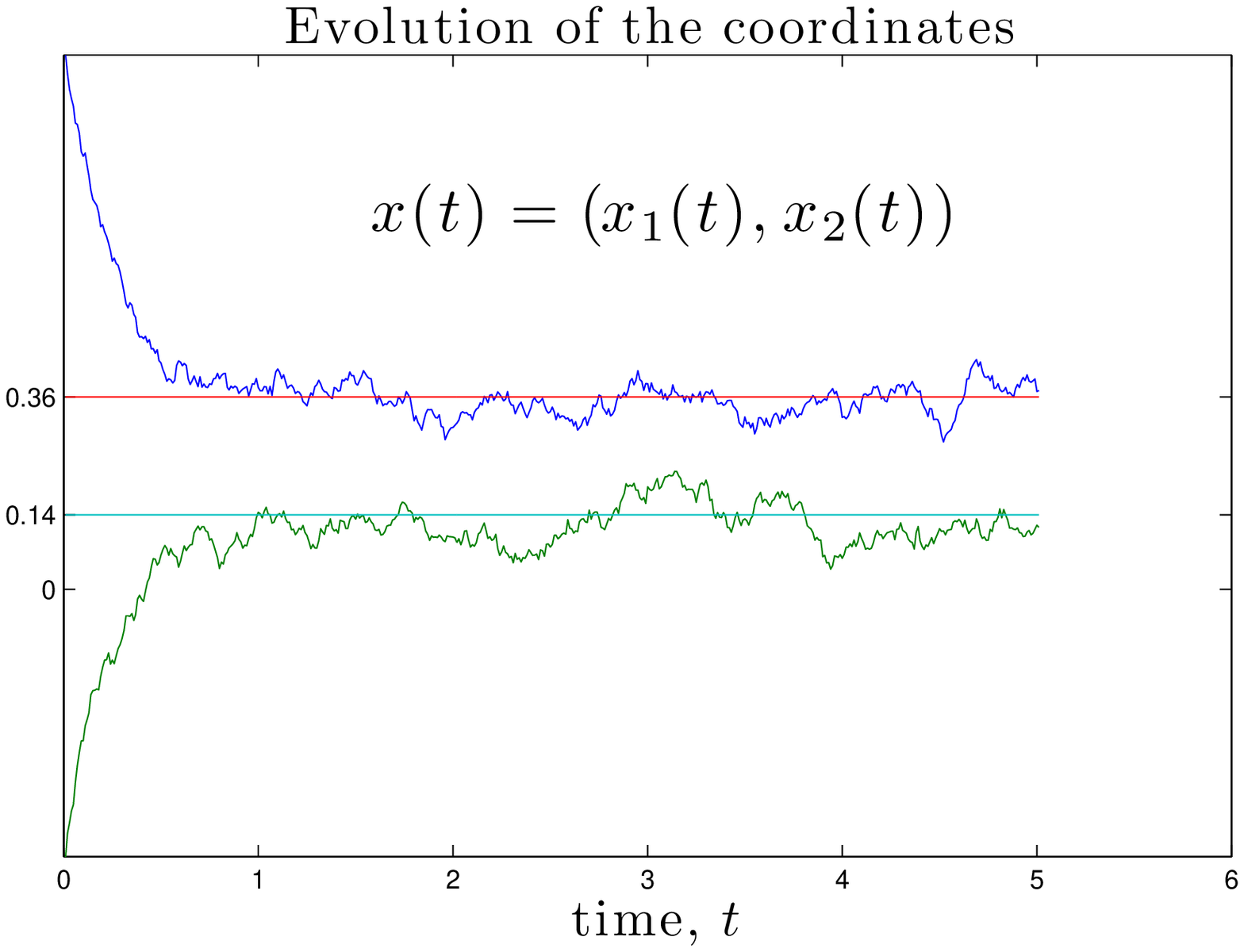}}
  \subfigure[]{\includegraphics[width=.496\linewidth]{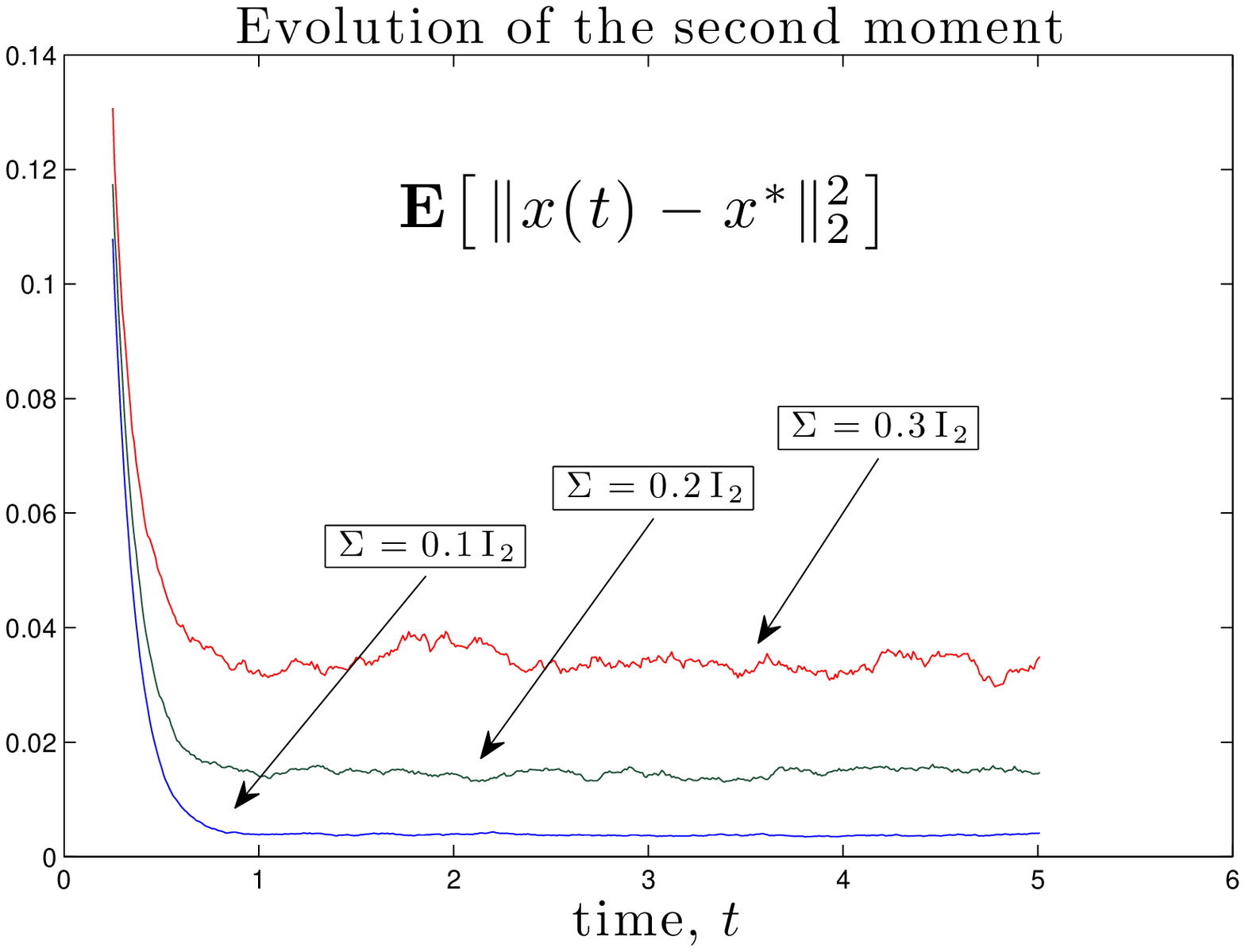}}
  \caption{Evolution of the
    dynamics~\eqref{eq:inexact-distributed-opt} with $\lap=0$,
    \,$\convexf(x_1,x_2)=\log\big(e^{(x_1-2)} + e^{(x_2+1)}\big) + 0.5
    (x_1+x_2-1)^2 + (x_1-x_2)^2$, and initial condition $[x_1(0),
    x_2(0)]=(1,-0.5)$. Since $\convexf$ is a sum of convex functions,
    and the Hessian of the quadratic part of $\convexf$ has
    eigenvalues $\{2, 4\}$, we can take $\gamma$ given by
    $\gamma(r)=2\, r$, for $r\ge 0$.  Plot (a) shows the evolution of
    the first and second coordinates with $\Sigma=0.1\,\identity_2$.
    Plot (b) illustrates the noise-to-state stability property in
    second moment with respect to~$x^*=(0.36, 0.14)$, where the matrix
    $\Sigma(t)$ is a constant multiple of the identity.  (The
    expectation is computed averaging over $500$ realizations of the
    noise.)}\label{fig:simulation}
\end{figure*}
    
\begin{example}\longthmtitle{Example~\ref{ex:example}--revisited:
    illustration of %
    Corollary~\ref{co:of-the-main-theorem}}
  {\rm Consider again Example~\ref{ex:example}. Since $\lyap$ is a
    $p$th moment NSS-Lyapunov function
    for~\eqref{eq:inexact-distributed-opt} with respect to the
    point~$x^*$ for $p\in(0,2]$, as shown in
    Example~\ref{ex:example-two},
    Corollary~\ref{co:of-the-main-theorem} implies that
    \begin{align}\label{eq:example-NSS-expectation}
      \expectationarg{\norm{x-x^*}^p} & \le \mu\big(\norm{x_0-x^*},
      t-t_0\big) + \theta\Big(\max_{t_0\le s\le
        t}\normfrob{\Sigma(s)}\Big) ,
    \end{align}
    for all $t\ge t_0$, $x_0\in\real^n$, and $p\in(0,2]$,
    where
    \begin{align*} 
      \mu(r,s) \:=\big(2 \muISStilde(r^2,s ) \big)^{p/2},
      \quad\theta(r)\: = \big(\gamma^{-1}(r^2)\big)^{p/2},
    \end{align*}
    and the class~$\classkl$ function $\muISStilde$ is defined as the
    solution to the initial value
    problem~\eqref{eq:class-kl-beta-satisfies} with
    $\increasingw(r)=\tfrac{1}{2}\gamma^{-1}(r)$.
    Figure~\ref{fig:simulation} illustrates this noise-to-state
    stability property.  We note that if the function $\convexf$ is
    strongly convex, i.e., if $\gamma(r)=c_{\gamma}\,r$ for some
    constant $c_{\gamma}>0$, then $\map{\muISStilde}{\realnonnegative
      \times \realnonnegative}{\realnonnegative}$ becomes
    $\muISStilde(r, s)=r e^{-c_{\gamma}s}$, and $\muISS(r,
    s)=2^{p/2}\,r^p e^{-c_{\gamma} p/2\, s}$, so the bound for
    $\expectationarg{\norm{x-x^*}^p}$
    in~\eqref{eq:example-NSS-expectation} decays exponentially with
    time to~$\theta\big(\max_{t_0\le s\le
      t}\normfrob{\Sigma(s)}\big)$.\oprocend}
\end{example}

\section{Refinements of the notion of proper
  functions}\label{sec:positive semidefinite-functions}

In this section, we analyze in detail the inequalities between
functions that appear in the definition of noise-dissipative Lyapunov
function, strong NSS-Lyapunov function in probability, and
$p$th~moment NSS-Lyapunov function.
In Section~\ref{sec:equivalent-classes}, we establish that these
inequalities can be regarded as equivalence relations. In
Section~\ref{sec:characterizations}, we make a complete
characterization of the properties of two functions related by these
equivalence relations. Finally, in Section~\ref{sec:alternative},
these results lead us to obtain an alternative formulation of
Corollary~\ref{co:of-the-main-theorem}.

\subsection{Proper functions and equivalence
  relations}\label{sec:equivalent-classes}

Here, we provide a refinement of the notion of proper functions with
respect to each other. Proper functions play an important role in
stability analysis, see e.g.,~\cite{HKK:02,EDS:08}.

\begin{definition}\longthmtitle{Refinements of the notion of proper
    functions with respect to each
    other}\label{def:kinf-proper-wrt-each-other}
  Let $\dom\subseteq\real^n$ and the functions $\map{\lyap,
    \lyapw}{\dom}{\realnonnegative}$ be such that
  \begin{align*} 
    \alpha_1 (\lyapw(x)) \le \lyap(x) \le \alpha_2 (\lyapw(x)), \quad
    \forall x \in \dom,
  \end{align*}
  for some functions $\map{\alpha_1,\,
    \alpha_2}{\realnonnegative}{\realnonnegative}$. Then,
  \begin{enumerate}
  \item if $\alpha_1,\alpha_2\in\classk$, we say that $\lyap$ is
    $\classk$\,-\,\emph{dominated by}~$\lyapw$ in $\dom$, and write
    $\linebreak\lyap\proper\lyapw\;\;\text{in}\;\;\dom$;
    
  \item if $\alpha_1,\alpha_2\in\classkinfty$, we say that $\lyap$ and
    $\lyapw$ are $\classkinfty$\,-\,\emph{proper with respect to each
      other in} $\dom$, and write $
    \lyap\properinfty\lyapw\;\;\text{in}\;\;\dom$;
    
  \item if $\alpha_1,\alpha_2\in\classkinfty$ are convex and concave,
    respectively, we say that $\lyap$ and $\lyapw$ are
    $\classkinftyccp$\,-\,\emph{(convex-concave) proper with respect
      to each other in} $\dom$, and write $ \lyap
    \properinftyccp\lyapw\;\;\text{in}\;\;\dom$;
    
  \item if $\alpha_1(r)\triangleq c_{\alpha_1} r$ and
    $\alpha_2(r)\triangleq c_{\alpha_2} r$, 
    for some constants
    $c_{\alpha_1}, c_{\alpha_2} > 0$, we say that $\lyap$ and $\lyapw$
    are \emph{equivalent in $\dom$}, and write 
    $\lyap\relationconstants\lyapw\;\;\text{in}\;\;\dom$.
  \end{enumerate}
\end{definition}

Note that the relations in
Definition~\ref{def:kinf-proper-wrt-each-other} are nested, i.e.,
given $\map{\lyap, \lyapw}{\dom}{\realnonnegative}$, the following
chain of implications hold in~$\dom$\,:
\begin{align}\label{eq:chain-relations}
  \lyap\relationconstants\lyapw \,\Rightarrow\,
  \lyap\properinftyccp\lyapw\,\Rightarrow\,\lyap
  \properinfty\lyapw\,\Rightarrow\,\lyap\proper\lyapw.
\end{align}
Also, note that if $\lyapw(x)=\norm{x}$,
$\dom$ is a neighborhood of~$0$, and $\alpha_1, \alpha_2$ are
class~$\classk$, then we recover the notion of $\lyap$ being a
\emph{proper} function~\cite{HKK:02}.  If $\dom=\real^n$, and $\lyap$
and $\lyapw$ are seminorms, then the relation~$\relationconstants$
corresponds to the concept of equivalent seminorms. 

The relation~$\properinfty$ is relevant for ISS and NSS in
probability, whereas the relation~$\properinftyccp$ is important for
$p$th moment NSS.  The latter is because the inequalities in
$\properinftyccp$ are still valid, thanks to Jensen inequality, if we
substitute $\lyap$ and $\lyapw$ by their expectations along a
stochastic process.  Another fact about the relation~$\properinftyccp$
is that $\alpha_1,\alpha_2\in\classkinfty$, convex and concave,
respectively, must be asymptotically linear if
$\lyap(\dom)\supseteq[s_0,\infty)$, for some $s_0\ge 0$, so that
$\alpha_1(s)\le\alpha_2(s)$ for all $s\ge s_0$. This follows from
Lemma~\ref{le:convex-concave}.

\begin{remark}\longthmtitle{Quadratic forms in a constrained
    domain}\label{re:another-set-assumptions}
  {\rm It is sometimes convenient to view the functions $\map{\lyap,
      \lyapw}{\dom}{\realnonnegative}$ as defined in a
    domain where their
    functional expression becomes simpler.  To make this idea precise,
    assume there exist $\map{i}{\dom \subset \real^n}{\real^m}$, with
    $m \ge n$, and 
    $\map{\lyaphat, \lyapwhat}{\hat{\dom}}{\realnonnegative}$, where
    $\hat{\dom}=i(\dom)$, such that $\lyap =\lyaphat \circ i $ and
    $\lyapw =\lyapwhat \circ i$.  If this is the case, then the
    existence of $\alpha_1,\,
    \alpha_2:\realnonnegative\to\realnonnegative$ such that
    $\alpha_1\big(\lyapwhat(\hat{x})\big) \le
    \lyaphat(\hat{x})\le\alpha_2\big(\lyapwhat(\hat{x})\big)$, for all
    $\hat{x}\in\hat{\dom}$, implies that
    $\alpha_1\big(\lyapw({x})\big)\le\lyap({x})
    \le\alpha_2\big(\lyapw({x})\big)$, for all $x\in\dom$. The reason
    is that for any $x\in\dom$ there exists $\hat{x}\in\hat{\dom}$,
    given by $\hat{x}=i(x)$, such that $\lyap(x)=\lyaphat(\hat{x})$
    and $\lyapw(x)=\lyapwhat(\hat{x})$, so
    \begin{align*}
      \alpha_1\big(\lyapw(x)\big) =
      \alpha_1\big(\lyapwhat(\hat{x})\big)\le\lyap(x) =
      \lyaphat(\hat{x})\le\alpha_2\big(\lyapwhat(\hat{x})\big) =
      \alpha_2\big(\lyapw(x)\big).
    \end{align*}
    Consequently, if any of the relations given in
    Definition~\ref{def:kinf-proper-wrt-each-other} is satisfied by
    $\lyaphat$, $\lyapwhat$ in $\hat{\dom}$, then the corresponding
    relation is satisfied by $\lyap$, $\lyapw$ in $\dom$. For
    instance, in some scenarios this procedure can allow us to rewrite
    the original functions $\lyap$, $\lyapw$ as quadratic forms
    $\lyaphat$, $\lyapwhat$ in a constrained set of an extended
    Euclidean space, where it is easier to establish the appropriate
    relation between the functions.  We make use of this observation
    in Section~\ref{sec:alternative} below.}\oprocend
\end{remark}

\begin{lemma}\longthmtitle{Powers
    of seminorms with the same
    nullspace}\label{le:kinf-proper-semidefinite-quadratic-forms}
  Let $A$ and $B$ in $\realmatricesrectangulararg{m}{n}$ be nonzero
  matrices with the same nullspace, $\kernel(A) = \kernel(B)$. Then,
  for any $p,q>0$, the inequalities
  $\alpha_1\big(\seminorm{x}{A}^p\big) \le \seminorm{x}{B}^q \le
  \alpha_2\big(\seminorm{x}{A}^p\big)$ are verified with
  \begin{align*}
    \alpha_1(r)\triangleq\Big(\frac{\lambda_{n-k}(\Bt
      B)}{\lambdamax(\At A)}\Big)^{\frac{q}{2}}\,
    r^{q/p};\quad\alpha_2(r)\triangleq \Big(\frac{\lambdamax(\Bt
      B)}{\lambda_{n-k}(\At A)}\Big)^{\frac{q}{2}}\, r^{q/p},
  \end{align*}
  where $k\triangleq\dim (\kernel(A))$.  In particular,
  $\seminorm{.}{A}^p\relationconstants\seminorm{.}{B}^p$\, and\,
  $\seminorm{.}{A}^p\properinfty\seminorm{.}{B}^q$\, in\, $\real^n$
  for any real numbers $p,q>0$.
\end{lemma}
\begin{proof}
  For $\Uset\triangleq\kernel(A)$,
  write any $x\in\real^n$ as $x=\xu+\xuperpendicular$, where
  $\xu\in\Uset$ and $\xuperpendicular\in\setdef{x\in\real^n}{x\tp
    u=0\;, \forall u\in\Uset}$, so that
  $Ax=A(\xu+\xuperpendicular)=A\xuperpendicular$ and
  $Bx=B\xuperpendicular$ because $\kernel(A)=\kernel(B)=\Uset$.  Using
  the formulas for the eigenvalues in~\cite[p. 178]{RAH-CRJ:85}, we
  see that the next chain of inequalities hold:
  \begin{align*}
    &\alpha_1\big(\seminorm{x}{A}^p\big)
    =\alpha_1\Big(\big(\xuperpendicular\tp \At
    A\xuperpendicular\big)^{\frac{p}{2}}\Big) \le \alpha_1
    \Big(\big(\lambdamax(\At A)\xuperpendicular\tp
    \xuperpendicular\big)^{\frac{p}{2}}\Big)
    \\
    &\le \big(\lambda_{n-k}(\Bt B) \xuperpendicular\tp
    \xuperpendicular\big)^{\frac{q}{2}} \le\big(\xuperpendicular\tp
    \Bt B\xuperpendicular\big)^{\frac{q}{2}} \le\big(\lambdamax(\Bt B)
    \xuperpendicular\tp \xuperpendicular\big)^{\frac{q}{2}}
    \\
    &\le \alpha_2 \Big(\big(\lambda_{n-k}(\At A) \xuperpendicular\tp
    \xuperpendicular\big)^{\frac{p}{2}}\Big)
    \le\alpha_2\Big(\big(\xuperpendicular\tp \At
    A\xuperpendicular\big)^{\frac{p}{2}}\Big)
    =\alpha_2\big(\seminorm{x}{A}^p\big),
  \end{align*}
  where $\seminorm{x}{B}^q=\big(\xuperpendicular\tp \Bt
  B\xuperpendicular\big)^{\frac{q}{2}}$. From this we conclude that
  $\seminorm{.}{A}^p\properinfty\seminorm{.}{B}^q$ in $\real^n$. Finally, when
  $p=q$, the class~$\classkinfty$ functions $\alpha_1$, $\alpha_2$
  in the statement are linear, so we obtain that
  $\seminorm{.}{A}^p\relationconstants\seminorm{.}{B}^p$ in $\real^n$.
\end{proof}

%
Next we show that~$\properinfty$ and~$\properinftyccp$ are reflexive,
symmetric, and transitive, and hence define equivalence relations.

\begin{lemma}\longthmtitle{The $\classkinfty$- and
    $\classkinftyccp$-proper relations are equivalence
    relations}\label{le:kinf-equivalence-relation}
  The relations $\properinfty$ and $\properinftyccp$ in any set
  $\dom\subseteq\real^n$ are both equivalence relations.
\end{lemma}
\begin{proof}
  For convenience, we represent both relations by~$\relation$.  Both
  are {reflexive}, i.e., $\lyap\relation\lyap$, because one can take
  $\alpha_1(r)=\alpha_2(r)=r$ noting that a linear function is both
  convex and concave.  Both are {symmetric}, i.e.,
  $\lyap\relation\lyapw$ if and only if $\lyapw\relation\lyap$,
  because if $\alpha_1 \circ \lyapw \le \lyap \le \alpha_2 \circ
  \lyapw$ in $\dom$, then $\alpha_2^{-1} \circ \lyap \le \lyapw \le
  \alpha_1^{-1} \circ \lyap$ in $\dom$. In the case of~$\properinfty$,
  the inverse of a class~$\classkinfty$ function is
  class~$\classkinfty$. Additionally, in the case
  of~$\properinftyccp$, if $\alpha\in\classkinfty$ is convex
  (respectively, concave), then $\alpha^{-1}\in\classkinfty$ is
  concave (respectively, convex). Finally, both are {transitive},
  i.e., $\lyapu\relation\lyap$ and $\lyap\relation\lyapw$ imply $
  \lyapu\relation\lyapw$, because if $\alpha_1 \circ \lyap \le \lyapu
  \le \alpha_2 \circ \lyap$ and $\tilde{\alpha}_1 \circ \lyapw \le
  \lyap \le \tilde{\alpha}_2 \circ \lyapw$ in $\dom$, then $\alpha_1
  \circ \tilde{\alpha}_1 \circ \lyapw \le \lyapu \le \alpha_2 \circ
  \tilde{\alpha}_2 \circ \lyapw$ in $\dom$. In the case
  of~$\properinfty$, the composition of two class $\classkinfty$
  functions is class~$\classkinfty$. Additionally, in the case
  of~$\properinftyccp$, if $\alpha_1,\alpha_2\in\classkinfty$ are both
  convex (respectively, concave), then the compositions
  $\alpha_1\circ\alpha_2$ and $\alpha_2\circ\alpha_1$ belong
  to~$\classkinfty$ and are convex (respectively, concave), as
  explained in Section~\ref{sec:pre-classks-convexity}.
\end{proof}

\begin{remark}\longthmtitle{The relation~$\proper$ is not an
    equivalence relation}
  {\rm The proof above also shows that the relation~$\proper$ is
    reflexive and transitive. However, it is not symmetric: consider
    $\lyap, \lyapw\in\psd$ given by $\lyap(x)=1-e^{-\norm{x}}$ and
    $\lyapw(x)=\norm{x}$. Clearly, $\lyap\proper\lyapw$ in $\real^n$
    by taking $\alpha_1=\alpha_2=\alpha \in\classk$, with
    $\alpha(s)=1-e^{-s}$. On the other hand, if there exist
    $\tilde{\alpha}_1, \tilde{\alpha}_2\in\classk$ such that
    $\tilde{\alpha}_1(\lyap(x)) \le \lyapw(x) \le
    \tilde{\alpha}_2(\lyap(x))$ for all $x\in\real^n$, then we reach
    the contradiction, by continuity of $\tilde{\alpha}_2$, that
    $\lim_{\norm{x}\to\infty}\norm{x} \le\tilde{\alpha}_2
    \big(\lim_{\norm{x}\to\infty} \big(
    1-e^{-\norm{x}}\big)\big)=\tilde{\alpha}_2(1)<\infty$.
  } \oprocend
\end{remark}

\subsection{Characterization of proper functions with respect to each
  other}\label{sec:characterizations}

In this section, we provide a complete characterization of the
properties that two functions must satisfy to be related by the
equivalence relations defined in Section~\ref{sec:equivalent-classes}.
For $\dom\subseteq \real^n$, consider 
$\map{\lyapone, \psdfuncalone}{\dom}{\realnonnegative}$. 
Given a real number~$p>0$, define
\begin{align*}
  \phip(s)&\triangleq\sup_{\{y\in\dom\,:\, \psdfunc{y}\,\le \sqrt[p]{s}\}}
  \lyapone(y),
  \\
  \psip(s)&\triangleq\inf_{\{y\in\dom\,:\, \psdfunc{y}\,\ge\sqrt[p]{s}\}}
  \lyapone(y),
\end{align*}
for $s \ge 0$. The value $\phip(s)$ gives the supremum of the function
$\lyapone$ in the $\sqrt[p]{s}$\,-\,sublevel set of $\psdfuncalone$,
and $\psip(s)$ is the infimum of $\lyapone$ in the
$\sqrt[p]{s}$\,-\,superlevel set of $\psdfuncalone$. Thus, the
functions $\phip$ and $\psip$ satisfy
\begin{align}\label{eq:psi-phi-sandwich}
  \psip\big(\psdfunc{x}^p\big)\,
  =\,\!\!\!\inf_{\substack{\{y\in\dom\,:\\ \psdfunc{y}\,\ge\,
      \psdfunc{x}\}}}\!\!\!\lyapone(y)\, \le\,\lyapone(x)\,
  \le\,\!\!\!\sup_{\substack{\{y\in\dom\,:\\ \psdfunc{y}\,\le\,
      \psdfunc{x}\}}}\!\!\!\lyapone(y)
  \,=\,\phip\big(\psdfunc{x}^p\big),
\end{align}
for all $x\in\dom$, which suggests $\phip$ and $\psip$ as
pre-comparison functions to construct $\alpha_1$ and $\alpha_2$ in
Definition~\ref{def:kinf-proper-wrt-each-other}.  To this end, we find
it useful to formulate the following properties of the function
$\lyapone$ with respect to $\psdfuncalone$:

%
$\hzero$: The set
$\{x\in\dom:\psdfunc{x} = s\}$ is nonempty for all $s\ge0$.

$\hone$: The nullsets of $\lyapone$ and $\psdfuncalone$ are the same,
i.e.,
$\setdef{x\in\dom}{\lyapone(x)=0}=\setdef{x\in\dom}{\psdfunc{x}=0}$.

$\htwo$: The function $\phione$ is locally bounded in
$\realnonnegative$ and right continuous at~$0$, and $\psione$ is
positive definite.

$\hthree$: The next limit holds: $\lim_{s\to\infty}\psione(s)=\infty$.

$\hfour$ (as a function of $p>0$): The asymptotic behavior of
$\phip$ and $\psip$ is such that $\phip(s)$ and $s^2/\psip(s)$ are
both in $\Obig(s)$ as $s\to\infty$.

The next result shows that these properties completely characterize
whether the functions $\lyapone$ and $\lyaptwo$ are related through the
equivalence relations defined in
Section~\ref{sec:equivalent-classes}. This result
generalizes~\cite[Lemma 4.3]{HKK:02} in several ways: the notions of
proper functions considered here are more general and are not necessarily
restricted to a relationship between an arbitrary function and the
distance to a compact set.

\begin{theorem}\longthmtitle{Characterizations of proper functions
    with respect to each other}\label{prop:V-is-kinf-proper-seminorm}
  Let $\map{\lyapone, \psdfuncalone}{\dom}{\realnonnegative}$, and
  assume $\psdfuncalone$ satisfies $\hzero$. Then
  \begin{enumerate}
  \item $\lyapone$ satisfies $\hipartialset$ with respect to
    $\psdfuncalone$ \;$\Leftrightarrow$\;
    $\lyapone\,\proper\,\psdfuncalone$ in $\dom$\,;
  
  \item $\lyapone$ satisfies $\hiset$ with respect to $\psdfuncalone$
    \;$\Leftrightarrow$\; $\lyapone\properinfty\psdfuncalone$ in
    $\dom$\,;
  
  \item $\lyapone$ satisfies $\hisetfull$ with respect to
    $\psdfuncalone$ for $p>0$ \;$\Leftrightarrow$\;
    $\lyapone\properinftyccp\psdfuncalone^p$ in $\dom$.
  \end{enumerate}
\end{theorem}
\begin{proof}
  We begin by establishing a few basic facts about the pre-comparison
  functions $\psip$ and $\phip$.
  By definition and by $\hzero$, it follows that $0\le\psione(s)\le
  \phione(s)$ for all $s\ge0$. Since $\phione$ is locally bounded by
  $\htwo$, then so is $\psione$. In particular, $\phione$ and
  $\psione$ are well defined in $\realnonnegative$. Moreover, both
  $\phione$ and $\psione$ are nondecreasing because if $s_2\ge s_1$,
  then the supremum is taken in a larger set,
  $\{x\in\dom:\psdfunc{x}\le s_2\}\supseteq\{x\in\dom:\psdfunc{x}\le
  s_1\}$, and the infimum is taken in a smaller set,
  $\{x\in\dom:\psdfunc{x}\ge s_2\}\subseteq\{x\in\dom:\psdfunc{x}\ge
  s_1\}$.
  Furthermore, for any $q>0$, the functions $\phiq$ and $\psiq$ are
  also monotonic and positive definite because
  $\phiq(s)=\phione(\sqrt[q]{s})$ and $\psiq(s)=\psione(\sqrt[q]{s})$
  for all $s\ge0$.  We now use these properties of the pre-comparison
  functions to construct $\alpha_1$, $\alpha_2$ in
  Definition~\ref{def:kinf-proper-wrt-each-other} required by the
  implications from left to right in each statement.
%
   
  Proof of \emph{(i)} ($\Rightarrow$). To show the existence of
  $\alpha_2\in\classk$ such that $\alpha_2(s)\ge\phione(s)$ for all
  $s\in\realnonnegative$, we proceed as follows. Since $\phione$ is
  locally bounded and nondecreasing, given a strictly increasing
  sequence $\{b_k\}_{k\ge1}\subseteq\realnonnegative$ with
  $\lim_{k\to\infty}b_k=\infty$, we choose the sequence
  $\{M_k\}_{k\ge1}\subseteq\realnonnegative$, setting $M_{0}=0$, in
  the following way:
  \begin{align}\label{eq:def+Mk}
    M_k\triangleq\max\Big\{\sup_{s\in[0, b_k]} \phione(s)\,,\:
    M_{k-1}+1/k^2\Big\} =\max\big\{\phione(b_k)\,,\:
    M_{k-1}+1/k^2\big\}.
  \end{align}
  This choice guarantees that $\{M_k\}_{k\ge1}$ is strictly increasing
  and, for each~$k\ge 1$,
  \begin{align}\label{eq:pi_bound}
    0\le M_k-\phione(b_k)\le\sum_{i=1}^k \frac{1}{i^2}\le \pi^2/6.
  \end{align}
  Also, since~$\phione$ is right continuous at~$0$, we can choose
  $b_1>0$ such that there exists
  $\map{\alpha_2}{[0,b_1]}{\realnonnegative}$ continuous, positive
  definite and strictly increasing, satisfying that
  $\alpha_2(s)\ge\phione(s)$ for all $s\in[0, b_1]$ and with
  $\alpha_2(b_1)=M_2$. (This is possible because the only function
  that cannot be upper bounded by an arbitrary continuous function in
  some arbitrarily small interval $[0,b_1]$ is the function that has a
  jump at~$0$.)
  The rest of the construction is explicit. We
  define $\alpha_2$ as a piecewise linear function in $(b_1, \infty)$ in
  the following way: for each $k\ge2$, we define
  \begin{align*}
    \alpha_2(s)\triangleq\alpha_2(b_{k-1}) +
    \frac{M_{k+1}-\alpha_2(b_{k-1})}{b_{k}-b_{k-1}}\,(s-b_{k-1}),
    \qquad\forall s\in(b_{k-1}, b_{k}].
  \end{align*}
  The resulting $\alpha_2$ is continuous by construction. Also,
  $\alpha_2(b_1)=M_2$, so that, inductively, $\alpha_2(b_{k-1})=M_{k}$
  for $k\ge2$. Two facts now follow: first,
  $M_{k+1}-\alpha_2(b_{k-1})=M_{k+1}-M_{k}\ge 1/(k+1)^2$ for $k\ge2$,
  so $\alpha_2$ has positive slope in each interval $(b_{k-1}, b_{k}]$
  and thus is strictly increasing in $(b_1,\infty)$; second,
  $\alpha_2(s)>\alpha_2(b_{k-1})= M_k\ge \phione(b_{k})\ge\phione(s)$
  for all $s\in(b_{k-1}, b_{k}]$, for each $k\ge2$, so
  $\alpha_2(s)\ge\phione(s)$ for all $s\in(b_1,\infty)$.

  We have left to show the existence of $\alpha_1\in\classk$ such that
  $\alpha_1(s)\le\psione(s)$ for all $s\in\realnonnegative$. First,
  since $0\le\psione(s)\le \phione(s)$ for all $s\ge0$ by definition
  and by $\hzero$, using the sandwich theorem~\cite[p. 107]{JL-ML:88},
  we derive that $\psione$ is right continuous at~$0$ the same as
  $\phione$. In addition, since $\psione$ is nondecreasing, it can
  only have a countable number of jump discontinuities (none of them
  at~$0$). Therefore, we can pick $c_1>0$ such that a continuous and
  nondecreasing function $\psionetilde$ can be constructed in $[0,
  c_1)$ by removing the jumps of $\psione$, so that
  $\psionetilde(s)\le\psione(s)$. Moreover, since $\psione$ is
  positive definite and right continuous at~$0$, then $\psionetilde$
  is also positive definite. Thus, there exists $\alpha_1$ in $[0,
  c_1)$ continuous, positive definite, and strictly increasing, such
  that, for some $r<1$,
  \begin{align}\label{eq:def-alpha1-before-c1}
    \alpha_1(s)\le r\psionetilde(s)\le
    r\psione(s) 
  \end{align} 
  for all $s\in [0, c_1)$. To extend $\alpha_1$ to a function in
  class~$\classk$ in~$\realnonnegative$, we follow a similar strategy
  as for $\alpha_2$. Given a strictly increasing sequence
  $\{c_k\}_{k\ge 2}\subseteq\realnonnegative$ with $\lim_{k\to\infty}
  c_k=\infty$, we define a sequence
  $\{m_k\}_{k\ge1}\subseteq\realnonnegative$ in the following way:
  \begin{align}\label{eq:def+mk}
    m_k\triangleq\inf_{s\in [c_{k}, c_{k+1})} \psione(s)
    -\tfrac{\psione(c_1)-\alpha_1(c_1)}{1+k^2}
    =\psione(c_{k})-\tfrac{\psione(c_1)-\alpha_1(c_1)}{1+k^2}.
  \end{align}  
  Next we define $\alpha_1$ in
  $[c_1, \infty)$ as the piecewise linear function
  \begin{align*} 
       \alpha_1(s)\triangleq\alpha_1(c_{k}) +
       \frac{m_{k}-\alpha_1(c_{k})}{c_{k+1}-c_{k}}\,(s-c_{k}), \qquad\forall
       s\in[c_{k}, c_{k+1}),
  \end{align*}
  for all $k\ge 1$, so $\alpha_1$ is continuous by construction. It is
  also strictly increasing because
  $\alpha_1(c_2)=m_1=(\psione(c_1)+\alpha_1(c_1))/2>\alpha_1(c_1)$
  by~\eqref{eq:def-alpha1-before-c1}, and also, for each $k\ge 2$, the
  slopes are positive because $m_{k}-\alpha_1(c_{k})=m_{k}-m_{k-1}>0$
  (due to the fact that $\{m_k\}_{k\ge 1}$ in~\eqref{eq:def+mk} is
  strictly increasing because $\psione$ is nondecreasing). Finally,
  $\alpha_1(s)<\alpha_1(c_{k+1})= m_k<\psione(c_k)\le \psione(s)$ for
  all $s\in[c_{k}, c_{k+1})$, for all $k\ge 1$ by~\eqref{eq:def+mk}.

  Equipped with $\alpha_1$, $\alpha_2$ as defined above, and as a
  consequence of~\eqref{eq:psi-phi-sandwich}, we have that
  \begin{align}\label{eq:application-psi-phi-sandwich}
    \alpha_1(\psdfunc{x})&\,\le\psione(\psdfunc{x}) \le \lyapone(x)
    \le\phione(\psdfunc{x})\le\alpha_2(\psdfunc{x}),\quad\forall
    x\in\dom.
  \end{align}
  This concludes the proof of~\emph{(i)}~($\Rightarrow$).
  
  As a preparation for \emph{(ii)-(iii)}~($\Rightarrow$), and
  assuming~$\hthree$, we derive two facts regarding the functions
  $\alpha_1$ and $\alpha_2$ constructed above.
  First, we establish that
  \begin{align}\label{eq:asymptotic-phione-alpha}
    \alpha_2(s)\in\Obig(\phione(s))\;\,\text{as}\;\, s\to\infty.
  \end{align}
  To show this, we argue that
  \begin{align}\label{eq:ingredient-slackness}
    \lim_{k\to\infty}\,\sup_{s\in(b_{k-1}, b_{k}]}
    \big(\alpha_2(s)-\phione(s)\big) \le\lim_{k\to\infty}
    \big(\phione(b_{k+1})-\phione(b_{k-1})\big)+\pi^2/6,
  \end{align}
  so that there exist $C, s_1>0$ such that $\alpha_2(s)\le
  3\phione(s)+C$, for all $s\ge s_1$. Thus, noting that
  $\lim_{s\to\infty} \phione(s)=\infty$ as a consequence of $\hthree$,
  the expression~\eqref{eq:asymptotic-phione-alpha} follows.
  To establish~\eqref{eq:ingredient-slackness}, we use the
  monotonicity of $\alpha_2$ and $\phione$,~\eqref{eq:def+Mk}
  and~\eqref{eq:pi_bound}. For $k\ge 2$,
  \begin{align*}
    &\sup_{s\in(b_{k-1}, b_{k}]}
    \big(\alpha_2(s)-\phione(s)\big)\le\alpha_2(b_k)
    -\phione(b_{k-1})=M_{k+1}-\phione(b_{k-1})
    \\
    =\,&\max\big\{\phione(b_{k+1})-\phione(b_{k-1})\,,\:
    M_{k}+1/(k+1)^2-\phione(b_{k-1})\big\}
    \\
    \le\, &\max\big\{\phione(b_{k+1})-\phione(b_{k-1})\,,\:
    \phione(b_k)+\pi^2/6+1/(k+1)^2-\phione(b_{k-1})\big\}.
  \end{align*}
  Second,
  the construction of $\alpha_1$ guarantees that
  \begin{align}\label{eq:asymptotic-psione-alpha}
    \psione(s)\in\Obig(\alpha_1(s))\;\, \text{as}\;\, s\to\infty,
  \end{align}
  because, as we show next,
  \begin{align} \label{eq:ingredient-slackness-psi}
    \lim_{k\to\infty}\,\sup_{s\in[c_{k}, c_{k+1})}
    \big(\psione(s)-\alpha_1(s)\big) \le\lim_{k\to\infty}
    \big(\alpha_1(c_{k+2})-\alpha_1(c_{k})\big),
  \end{align}
  so there exists $s_2>0$ such that $\psione(s)\le 3\alpha_1(s)$ for
  all $s\ge s_2$. To obtain~\eqref{eq:ingredient-slackness-psi}, we
  leverage the monotonicity of $\psione$ and $\alpha_1$,
  and~\eqref{eq:def+mk}; namely, for $k\ge2$,
  \begin{align*} 
    \,&\sup_{s\in[c_{k}, c_{k+1})} \big(\psione(s)-\alpha_1(s)\big)
    \le\psione(c_{k+1})-\alpha_1(c_{k}) \\=\,&
    m_{k+1}+\tfrac{\psione(c_1)-\alpha_1(c_1)}{1+(k+1)^2}-\alpha_1(c_{k})
    = \alpha_1(c_{k+2})+\tfrac{\psione(c_1)-\alpha_1(c_1)}{1+(k+1)^2} -
    \alpha_1(c_{k}).
  \end{align*}
  Equipped with~\eqref{eq:asymptotic-phione-alpha}
  and~\eqref{eq:asymptotic-psione-alpha}, we prove next
  \emph{(ii)-(iii)}~($\Rightarrow$).
  
  Proof of \emph{(ii)}~($\Rightarrow$): If, in addition,~$\hthree$ holds,
  then $\lim_{s\to\infty} \phione(s)\ge\lim_{s\to\infty}
  \psione(s)=\infty$. This guarantees that
  $\alpha_2\in\classkinfty$. Also, according
  to~\eqref{eq:asymptotic-psione-alpha}, $\hthree$ implies that
  $\alpha_1$ is unbounded, and thus in $\classkinfty$ as well.  The
  result now follows by~\eqref{eq:application-psi-phi-sandwich}.
      
  Proof of \emph{(iii)}~($\Rightarrow$): Finally, assume that $\hfour$
  also holds for some $p>0$. We show next the existence of the
  required convex and concave functions involved in the
  relation~$\properinftyccp$.
  %
  Let $\alpha_{1,p}(s)\triangleq\alpha_1(\sqrt[p]{s})$ and
  $\alpha_{2,p}(s)\triangleq\alpha_2(\sqrt[p]{s})$ for $s\ge 0$, so
  that
  \begin{align*}
    \alpha_{1,p}(s)=\alpha_1(\sqrt[p]{s})\le\psione(\sqrt[p]{s})=\psip(s)
    \quad\text{and}\quad
    \phip(s)=\phione(\sqrt[p]{s})\le\alpha_2(\sqrt[p]{s})=\alpha_{2,p}(s).
  \end{align*}
  From~\eqref{eq:asymptotic-phione-alpha} and~$\hfour$, it follows
  that there exist $s'$, $c_1$, $c_2>0$ such that $\alpha_2(s)\le c_1
  \phione(s)$ and $\phip(s)\le c_2 s$ for all $s\ge s'$. Thus,
  \begin{align*}
    \alpha_{2,p}(s)=\alpha_2(\sqrt[p]{s})\le c_1
    \phione(\sqrt[p]{s})=c_1 \phip(s)\le c_1 c_2 s,
  \end{align*}
  for all $s\ge s'$, so $\alpha_{2,p}(s)$ is in $\Obig(s)$ as
  $s\to\infty$.  Similarly, according to~\eqref{eq:asymptotic-psione-alpha}
  and~$\hfour$, there are constants $s''$, $c_3$, $c_4>0$ such
  that $\psione(s)\le c_3\, \alpha_1(s)$ and $s^2\le c_4 s\, \psip(s)$
  for all $s\ge s''$. Thus,
  \begin{align*}
    s\, \alpha_{1,p}(s)=s\, \alpha_1(\sqrt[p]{s})\ge s\,\tfrac{1}{c_3}
    \psione(\sqrt[p]{s}) =s\,\tfrac{1}{c_3} \psip(s)\ge \tfrac{1}{c_3
      c_4} s^2,
  \end{align*}
  for all $s\ge s''$, so $s^2/\alpha_{1,p}(s)$ is in $\Obig(s)$ as
  $s\to\infty$. Summarizing, the construction of $\alpha_1$,
  $\alpha_2$ guarantees, under $\hfour$, that $\alpha_{1,p}$,
  $\alpha_{2,p}$ satisfy that $s^2/\alpha_{1,p}(s)$ and
  $\alpha_{2,p}(s)$ are in $\Obig(s)$ as $s\to\infty$ (and, as a
  consequence, so are $s^2/\alpha_{2,p}(s)$ and
  $\alpha_{1,p}(s)$). Therefore, according to
  Lemma~\ref{le:convex-concave}, we can
  leverage~\eqref{eq:application-psi-phi-sandwich} by taking
  $\tilde{\alpha}_1$, $\tilde{\alpha}_2\in\classkinfty$, convex
    and concave, respectively, such that, for
  all~$x\in\dom$,
  \begin{align*}
    \tilde{\alpha}_1\big(\psdfunc{x}^p\big)
    \le&\,\alpha_{1,p}(\psdfunc{x}^p) =\alpha_1(\psdfunc{x})
    \le
    \psione\big(\psdfunc{x}\big)\le\lyapone(x)
    \\
    \le&\,\phione\big(\psdfunc{x}\big) 
    \le\alpha_2(\psdfunc{x}) =\alpha_{2,p}(\psdfunc{x}^p)
    \le\tilde{\alpha}_2\big(\psdfunc{x}^p\big).
  \end{align*}
  
  Proof of \emph{(i)}~($\Leftarrow$): If there exist class~$\classk$
  functions $\alpha_1$, $\alpha_2$ such that $\alpha_1(\psdfunc{x})\le
  \lyapone(x)\le\alpha_2(\psdfunc{x})$ for all $x\in\dom$, then the
  nullsets of $\lyapone$ and $\psdfuncalone$ are the same, which is
  the property~$\hone$. In addition,
  $0\le\phione(s)\le\alpha_2(s)$ for all $s\ge 0$, so $\phione$ is
  locally bounded and, moreover, the sandwich theorem guarantees that
  $\phione$ is right continuous at~$0$. Also, since
  $\alpha_1(s)\le\psione(s)$, for all $s\ge 0$, and $\psione(0)=0$, it
  follows that $\psione$ is positive definite. Therefore, $\htwo$ also
  holds.
    
  Proof of \emph{(ii)}~($\Leftarrow$): Since
  $\psione(s)\ge\alpha_1(s)$ for all $s\ge 0$, the property~$\hthree$
  follows because
  \begin{align*}
    \lim_{s\to\infty}\psione(s)\ge\lim_{s\to\infty}\alpha_1(s)=\infty.
  \end{align*}
  
  Proof of \emph{(iii)}~($\Leftarrow$): If
  $\lyapone\properinftyccp\psdfuncalone^p$, then
  $\lyapone\properinfty\psdfuncalone^p$
  by~\eqref{eq:chain-relations}. Also, we have trivially that
  $\psdfuncalone^p\properinfty\psdfuncalone$. Since~$\properinfty$ is
  an equivalence relation by Lemma~\ref{le:kinf-equivalence-relation},
  it follows that $\lyapone\properinfty\psdfuncalone$, so the
  properties $\hiset$ hold as in \emph{(ii)}~($\Leftarrow$). We have
  left to derive $\hfour$. If
  $\lyapone\properinftyccp\psdfuncalone^p$, then there exist
  $\alpha_1,\alpha_2\in\classkinfty$ convex and concave, respectively,
  such that $\alpha_1\big(\psdfunc{x}^p\big)\le\lyapone(x) \le
  \alpha_2\big(\psdfunc{x}^p\big)$ for all $x\in\dom$.  Hence, by the
  definition of $\psip$ and $\phip$, and~$\hzero$, and by the
  monotonicity of $\alpha_1$ and $\alpha_2$, we have that, for all
  $s\ge 0$,
  \begin{align}\label{eq:psis-and-alphas}
    \alpha_1(s) & \le \inf_{\{x\in\dom\,:\,\psdfunc{x}^p\ge s\}}
    \alpha_1\big(\psdfunc{x}^p\big)
    \le\inf_{\{x\in\dom\,:\,\psdfunc{x}^p\ge s\}} \lyapone(x)
    =\psip(s)\nonumber
    \\
    & \le\, \phip(s) =\sup_{\{x\in\dom\,:\, \psdfunc{x}^p\le
      s\}}\lyapone(x) \le\sup_{\{x\in\dom\,:\, \psdfunc{x}^p\le
      s\}}\alpha_2\big(\psdfunc{x}^p\big) \le\alpha_2(s).
  \end{align}
  Now, since $\alpha_1$,$\alpha_2\in\classkinfty$ are convex and
  concave, respectively, it follows by Lemma~\ref{le:convex-concave}
  that $s^2/\alpha_1(s)$ and $\alpha_2(s)$ are in $\Obig(s)$ as
  $s\to\infty$.  Knowing from~\eqref{eq:psis-and-alphas} that
  $\alpha_1(s)\le \psip(s)\le\phip(s)\le\alpha_2(s)$ for all $s\ge 0$,
  we conclude that the functions $s^2/\psip(s)$ and $\phip(s)$ are
  also in $\Obig(s)$ as $s\to\infty$, which is the property~$\hfour$.
\end{proof}

The following example shows ways in which the conditions of
Theorem~\ref{prop:V-is-kinf-proper-seminorm} might fail.

\begin{example}\longthmtitle{Illustration
    of~Theorem~\ref{prop:V-is-kinf-proper-seminorm}}
  {\rm Let $\map{\lyaptwo}{\real^2}{\realnonnegative}$ be the distance
    to the set $\setdef{(x_1,x_2)\in\real^2}{x_1=0}$, i.e.,
    $\lyaptwo(x_1,x_2)=\absolute{x_1}$.  Consider the following cases:

    \emph{P2 fails ($\psione$ is not positive definite):} Let
    $\lyapone(x_1,x_2)=\absolute{x_1}e^{-\absolute{x_2}}$ for
    $(x_1,x_2)\in\real^2$.  Note that $\lyapone$ is not
    $\classk$\,-\,dominated by $\psdfuncalone$ because, given any
    $\alpha_1\in\classk$, for every $x_1\in\real$ with
    $\absolute{x_1}>0$ there exists $x_2\in\real$ such that the
    inequality $\alpha_1(\absolute{x_1})\le
    \absolute{x_1}e^{-\absolute{x_2}}$ does not hold (just choose
    $x_2$ satisfying $\absolute{x_2} >
    \log\big(\tfrac{\absolute{x_1}}{\alpha_1(\absolute{x_1})}\big)$). Thus,
    there must be some of the hypotheses on
    Theorem~\ref{prop:V-is-kinf-proper-seminorm} that fail to be
    true. In this case, we observe that
    \begin{align*}
      \psione(s)=\inf_{\{(x_1,x_2)\in\real^2\,:\, \absolute{x_1}\,\ge
        s\}} \absolute{x_1}e^{-\absolute{x_2}}
    \end{align*}
    is identically~$0$ for all $s\ge 0$, so it is not positive
    definite as required in~$\htwo$.

    \emph{P2 fails ($\phione$ is not locally bounded):} Let
    $\lyapone(x_1,x_2)=\absolute{x_1}e^{\absolute{x_2}}$ for
    $(x_1,x_2)\in\real^2$. As above, one can show that $\alpha_2$ does
    not exist in the required class; in this case, the
    hypothesis~$\htwo$ is not satisfied because $\phione$ is not
    locally bounded in $(0,\infty)$:
    \begin{align*}
      \phione(s)=\sup_{\{(x_1,x_2)\in\real^2\,:\, \absolute{x_1}\,\le
        s\}} \absolute{x_1}e^{\absolute{x_2}}=\infty,  \quad \forall\;
      s>0.
    \end{align*}

    \emph{P2 fails ($\phione$ is not right continuous):} Let
    $\lyapone(x_1,x_2)=\absolute{x_1}^4 +\absolute{\sin{(x_1
        x_2)}}$ for $(x_1,x_2)\in\real^2$. 
    For every $p>0$, we have that
    \begin{align*}
      \phip(s)=\sup_{\{(x_1,x_2)\in\real^2\,:\, \absolute{x_1}^p\,\le
        s\}} \absolute{x_1}^4 +\absolute{\sin{(x_1 x_2)}}\le
      s^{4/p}+1,
    \end{align*}
    so $\phip$ is locally bounded in $\realnonnegative$, and, again
    for every $p>0$,
    \begin{align*}
      \psip(s)=\inf_{\{(x_1,x_2)\in\real^2\,:\, \absolute{x_1}^p\,\ge
        s\}} \absolute{x_1}^4 +\absolute{\sin{(x_1 x_2)}}\ge s^{4/p},
    \end{align*}
    so $\psip$ is positive definite.  However, $\phip$ is not right
    continuous at~$0$ because $\sin{(x_1 x_2)}=0$ when $x_1=0$, but
    $\sup_{\{(x_1,x_2)\in\real^2\,:\, \absolute{x_1}^p\,\le s_0\}}
    \sin{(x_1 x_2)}=1$ for any $s_0>0$, so by
    Theorem~\ref{prop:V-is-kinf-proper-seminorm}~\emph{(i)}, it
    follows that $\lyapone$ is not $\classk$\,-\,dominated by
    $\psdfuncalone$.

    \emph{P4 fails (non-compliant asymptotic behavior):} Let
    $\lyapone(x_1,x_2)=\absolute{x_1}^4$ for $(x_1,x_2)\in\real^2$. Then $\htwo$ is satisfied
    and $\hthree$ also holds because
    $\lim_{s\to\infty}\psione(s)=\lim_{s\to\infty} s^4=\infty$, so
    Theorem~\ref{prop:V-is-kinf-proper-seminorm}~\emph{(ii)} implies that
    $\lyapone$ and $\lyaptwo$ are $\classkinfty$-proper with respect
    to each other. However, in this case $\phip(s)=\psip(s)=s^{4/p}$,
    which implies that $\phip$ is not in $\Obig(s)$ as $s\to\infty$
    when $p\in(0,4)$, and $s^2/\psip(s)$ is not in $\Obig(s)$ as
    $s\to\infty$ when~$p>4$. Thus $\hfour$ is satisfied only for
    $p=4$, so Theorem~\ref{prop:V-is-kinf-proper-seminorm}~\emph{(iii)} implies
    that only in this case $\lyapone$ and $\lyaptwo^p$ are
    $\classkinftyccp$- proper with respect to each other.  Namely,
    for $p>4$, one cannot choose a convex
    $\alpha_1\in\classkinfty$ such that $\alpha_1(\absolute{x_1}^p)\le
    \absolute{x_1}^4$ for all $x_1\in\real$ and, if
    $p<4$, one cannot choose a concave $\alpha_2 \in
    \classkinfty$ such that $\absolute{x_1}^4 \le
    \alpha_2(\absolute{x_1}^p)$ for all $x_1\in\real$. 
    \oprocend}
\end{example}


\subsection{Application to noise-to-state
  stability}\label{sec:alternative}

In this section we use the results of
Sections~\ref{sec:equivalent-classes} and~\ref{sec:characterizations}
to study the noise-to-state stability properties of stochastic
differential equations of the
form~\eqref{eq:nonlinear-SDE-preliminaries}.  Our first result
provides a way to check whether a candidate function that satisfies a
dissipation inequality of the
type~\eqref{eq:theorem-second-hypothesis-VandW} is in fact a
noise-dissipative Lyapunov function, a strong NSS-Lyapunov function in
probability, or a $p$th~moment NSS-Lyapunov function.

\begin{corollary}\longthmtitle{Establishing proper relations between
    pairs of functions through
    seminorms}\label{co:kinf-proper-wrt-each-other} 
  Consider $\map{\lyap_1, \lyap_2}{\dom}{\realnonnegative}$ such that
  their nullset is a subspace~$\Uset$. Let $A,\Atilde \in
  \realmatricesrectangulararg{m}{n}$ be such that $\kernel(A) = \Uset
  = \kernel(\Atilde)$.  Assume that $\lyap_1$ and $\lyap_2$ satisfy
  $\{\operatorname{P}$i$\}_{i=0}^3$ with respect to $\seminorm{.}{A}$
  and $\seminorm{.}{\Atilde}$, respectively.  Then, for any $q>0$,
  \begin{align*}
    \lyap_1\properinfty\lyap_2, \quad \lyap_1\properinfty
    \seminorm{.}{A}^q,
    \quad
    \lyap_2\properinfty\seminorm{.}{\Atilde}^q\quad\text{in}\quad\dom.
  \end{align*}
  
  If, in addition, $\lyap_1$ and $\lyap_2$ satisfy $\hfour$ with
  respect to $\seminorm{.}{A}$ and $\seminorm{.}{\Atilde}$,
  respectively, for some $p>0$, then
  \begin{align*}
    \lyap_1\properinftyccp\lyap_2,
    \quad\lyap_1\properinftyccp\seminorm{.}{A}^p,
    \quad
    \lyap_2\properinftyccp\seminorm{.}{\Atilde}^p
    \quad\text{in}\quad\dom.
  \end{align*}
\end{corollary}
\begin{proof}
  The statements follow from the characterizations in
  Theorem~\ref{prop:V-is-kinf-proper-seminorm}~\emph{(ii)}~and~\emph{(iii)},
  and from the fact that the relations $\properinfty$ and
  $\properinftyccp$ are equivalence relations as shown in
  Lemma~\ref{le:kinf-equivalence-relation}. That is, under the
  hypothesis~$\hzero$,
  \begin{align*}
    \!\!\left.\begin{array}{c}
        \!\!\lyap_1\:\text{satisfies}\:\hiset\;\text{w/ respect to}\;
        \seminorm{.}{A}\; (\Leftrightarrow\;
        \lyap_1\properinfty\seminorm{.}{A}\;\text{in}\;\dom)
        \\
        \!\!\lyap_2\:\text{satisfies}\:\hisettilde\;\text{w/ respect to}\;
        \seminorm{.}{\Atilde}\; (\Leftrightarrow\;
        \lyap_2\properinfty\seminorm{.}{\Atilde}\;\text{in}\;\dom)
      \end{array}
      \!\!\right\}\Rightarrow
    \lyap_1\properinfty\lyap_2\;\text{in}\;\dom,
    \\
    \!\!\left.\begin{array}{c}
        \!\!\lyap_1\:\text{satisfies}\:\hisetfull\;\text{w/ respect to}\;
        \seminorm{.}{A}\; (\Leftrightarrow\;
        \lyap_1\properinftyccp\seminorm{.}{A}^p\;\text{in}\;\dom)
        \\
        \!\!\lyap_2\:\text{satisfies}\:\hisetfulltilde\;\text{w/
          respect to}\;
        \seminorm{.}{\Atilde}\; (\Leftrightarrow\;
        \lyap_2\properinftyccp\seminorm{.}{\Atilde}^p\;\text{in}\;\dom)
      \end{array}
      \!\!\right\}\Rightarrow
    \lyap_1\properinftyccp\lyap_2\;\text{in}\;\dom.
  \end{align*}
  Note that, by
  Lemma~\ref{le:kinf-proper-semidefinite-quadratic-forms}
  and~\eqref{eq:chain-relations}, the equivalences
  \begin{align*}
    \seminorm{.}{A}\properinfty\seminorm{.}{\Atilde}^q\;\;\;\text{in}\;\;\dom,
    \qquad
    \seminorm{.}{A}^p\properinftyccp\seminorm{.}{\Atilde}^p\;\;\;\text{in}\;\;
    \dom
  \end{align*}
  hold for any $p,q>0$ and any matrices
  $A,\,\Atilde\in\realmatricesrectangulararg{m}{n}$ with
  $\kernel(A)=\kernel(\Atilde)$.
\end{proof}

We next build on this result to provide an alternative formulation of
Corollary~\ref{co:of-the-main-theorem}. To do so, we employ the
observation made in Remark~\ref{re:another-set-assumptions} about the
possibility of interpreting the candidate functions as defined on a
constrained domain of an extended Euclidean space.

\begin{corollary}\longthmtitle{The existence of a $p$thNSS-Lyapunov
    function implies $p$th~moment NSS --revisited}\label{co:alternative}
  Under Assumption~\ref{ass:assumptions-SDE}, let $\lyap\in\psdtwice$,
  $\lyapw\in\psd$ and $\gainnoise\in\classk$ be such that the
  dissipation inequality~\eqref{eq:theorem-hypothesis-ito} holds.  Let
  $\map{R}{\real^n}{\real^{(m-n)}}$, with $m\ge n$, $\dom \subset
  \real^m$, $\lyaphat\in\psddomaintwice$ and $\lyapwhat\in\psddomain$
  be such that, for $i(x) = [x\tp,R(x)\tp]\tp$, one has
  \begin{align*}
    \dom = i(\real^n), \quad \lyap = \lyaphat \circ i, \quad
    \text{and} \quad \lyapw = \lyapwhat \circ i .
  \end{align*}
  Let $A = \diag(A_1, A_2)$ and $\Atilde = \diag(\Atilde_1,\Atilde_2)$
  be block-diagonal matrices, with $A_1, \Atilde_1\in\realmatrices$
  and $A_2, \Atilde_2\in\realmatricesrectangulararg{(m-n)}{(m-n)}$,
  such that $\kernel(A)=\kernel(\Atilde)$ and
  \begin{align}\label{eq:bound-R}
    \seminorm{R(x)}{A_2}^2\le \kappa\seminorm{x}{A_1}^2
  \end{align}
  for some $\kappa>0$, for all $x\in\real^n$.  Assume that $\lyaphat$
  and $\lyapwhat$ satisfy the properties
  $\{\operatorname{P}$i$\}_{i=0}^4$ with respect to $\seminorm{.}{A}$
  and $\seminorm{.}{\Atilde}$, respectively, for some $p>0$. Then the
  system~\eqref{eq:nonlinear-SDE-preliminaries} is NSS in probability
  and in $p$th~moment with respect to~$\kernel(A_1)$.
\end{corollary}
\begin{proof}
  By Corollary~\ref{co:kinf-proper-wrt-each-other}, we have that
  \begin{align}\label{eq:co-other-assumptions}
    \lyaphat\properinftyccp\lyapwhat, \quad \text{and} \quad
    \lyaphat\properinftyccp\seminorm{.}{\diag(A_1,
      A_2)}^p\quad\text{in}\quad\dom.
  \end{align}
  As explained in Remark~\ref{re:another-set-assumptions}, the first
  relation implies that $\lyap\properinftyccp\lyapw$ in
  $\real^n$. This, together with the fact
  that~\eqref{eq:theorem-hypothesis-ito} holds, implies that $\lyap$
  is a noise-dissipative Lyapunov function
  for~\eqref{eq:nonlinear-SDE-preliminaries}. Also, setting
  $\hat{x}=i(x)$ and using~\eqref{eq:bound-R}, we obtain that
  \begin{align*} 
    \seminorm{x}{A_1}^2\le\seminorm{\hat{x}}{\diag(A_1,
      A_2)}^2=\seminorm{x}{A_1}^2+\seminorm{R(x)}{A_2}^2\le
    (1+\kappa)\seminorm{x}{A_1}^2,
  \end{align*}  
  so, in particular, $\seminorm{[\,. ,R(.)]}{\diag(A_1,
    A_2)}^p\relationconstants\seminorm{.}{A_1}^p$ in~$\real^n$.
  Now, from the second relation in~\eqref{eq:co-other-assumptions}, by
  Remark~\ref{re:another-set-assumptions}, it follows that $\lyaphat
  \circ\, i\properinftyccp\seminorm{[\,.,R(.)]}{\diag(A_1, A_2)}^p$ in
  $\real^n$. Thus, using~\eqref{eq:chain-relations} and
  Lemma~\ref{le:kinf-equivalence-relation}, we conclude that
  $\lyap\properinftyccp\seminorm{.}{A_1}^p$ in $\real^n$. In addition,
  the Euclidean distance to the set $\kernel(A_1)$ is equivalent to
  $\seminorm{.}{A_1}$, i.e.,
  $\distset{.}{\kernel(A_1)}\relationconstants\seminorm{.}{A_1}$. This
  can be justified as follows: choose
  $B\in\realmatricesrectangulararg{n}{k}$, with
  $k=\dim(\kernel(A_1))$, such that the columns of $B$ form an
  orthonormal basis of~$\kernel(A_1)$. Then,
  \begin{align}
    \distset{x}{\kernel(A_1)}=\norm{(\identity -B
      B\tp)x}=\seminorm{x}{\identity -B
      B\tp}\relationconstants\seminorm{.}{A_1},
  \end{align}
  where the last relation follows from
  Lemma~\ref{le:kinf-proper-semidefinite-quadratic-forms} because
  $\kernel(\identity -B B\tp)=\kernel(A_1)$. Summarizing,
  $\lyap\properinftyccp\seminorm{.}{A_1}^p$ and
  $\seminorm{.}{A_1}^p\relationconstants \distset{x}{\kernel(A_1)}^p$
  in $\real^n$ (because the $p$th power is irrelevant for the
  relation~$\relationconstants$). As a consequence,
  \begin{align}\label{eq:proper-alternative-hypotheses}
    \lyap\properinftyccp\distset{.}{\kernel(A_1)}^p\quad
    \text{in}\quad\real^n,
  \end{align}
  which implies condition~\eqref{eq:theorem-third-hypothesis-V} with
  convex $\alpha_1\in\classkinfty$, concave $\alpha_2\in\classkinfty$,
  and $\Uset=\kernel(A_1)$.  Therefore, $\lyap$ is a $p$th moment
  NSS-Lyapunov function with respect to the set~$\kernel(A_1)$, and
  the result follows from Corollary~\ref{co:of-the-main-theorem}.
\end{proof}


\section{Conclusions}\label{sec:conclusions-future}

We have studied the stability properties of SDEs subject to persistent
noise (including the case of additive noise).  We have generalized the
concept of noise-dissipative Lyapunov function and introduced the
concepts of strong NSS-Lyapunov function in probability and
$p$th~moment NSS-Lyapunov function, both with respect to a closed set.
We have shown that noise-dissipative Lyapunov functions have NSS
dynamics and established that the existence of an NSS-Lyapunov
function, of either type, with respect to a closed set, implies the
corresponding NSS property of the system with respect to the set.  In
particular, $p$th moment NSS with respect to a set provides a bound,
at each time, for the $p$th~power of the distance from the state to
the set, and this bound is the sum of an increasing function of the
size of the noise covariance and a decaying effect of the initial
conditions.  This bound can be achieved regardless of the possibility
that inside the set some combination of the states accumulates the
variance of the noise. This is a meaningful stability property for the
aforementioned class of systems because the presence of persistent
noise makes it impossible to establish in general a stochastic notion
of asymptotic stability for the set of equilibria of the underlying
differential equation.  We have also studied in depth the inequalities
between pairs of functions that appear in the various notions of
Lyapunov functions mentioned above.  We have shown that these
inequalities define equivalence relations and have developed a
complete characterization of the properties that two functions must
satisfy to be related by them.  Finally, building on this
characterization, we have provided an alternative statement of our
stochastic stability results.
Future work will include the study of the effect of delays and
impulsive right-hand sides in the class of SDEs considered in this
paper.

\section*{Acknowledgments}
The first author would like to thank Dean Richert for useful
discussions. In addition,
the authors would like to thank Dr. Fengzhong Li for his kind observations that have made possible an important correction of the proof of Theorem~\ref{th:Stability-Non-Linear-Systems}. 
 The research was supported by NSF award CMMI-1300272.

\newcounter{mycounter}
\renewcommand{\themycounter}{A.\arabic{mycounter}}
\newtheorem{lemmaappendix}[mycounter]{Lemma}
\section*{Appendix}\label{sec:appendix}

The next result is used in the proof of
Theorem~\ref{prop:V-is-kinf-proper-seminorm}.

\begin{lemmaappendix}\longthmtitle{Existence of bounding convex and
    concave functions in $\classkinfty$}\label{le:convex-concave}
  Let $\alpha$ be a class~$\classkinfty$ function.  Then the following
  are equivalent:
  \begin{enumerate}
  \item There exist $s_0 \ge 0$ and $\alpha_1,
    \alpha_2\in\classkinfty$, convex and concave, respectively, such
    that $\alpha_1 (s) \le \alpha (s) \le \alpha_2 (s) $ for all $s
    \ge s_0$, and
    
  \item $\alpha(s), \,s^2/\alpha(s)$ are in $\Obig(s)$ as
    $s\to\infty$.
  \end{enumerate}
\end{lemmaappendix}
\begin{proof}
  %
  The implication $\emph{(i)}\Rightarrow\emph{(ii)}$ follows because,
  for any $s\ge s_0>0$,
  \begin{align*}
    \tfrac{\alpha_1(s_0)}{s_0} s\le\alpha_1(s)\le\alpha(s)
    \le\alpha_2(s)\le\tfrac{\alpha_2(s_0)}{s_0} s,
  \end{align*}
  by convexity and concavity, respectively, where
  $\alpha_1(s_0),\alpha_2(s_0)>0$.

  To show $\emph{(ii)}\Rightarrow\emph{(i)}$, we proceed to construct
  $\alpha_1,\alpha_2$ as in the statement using the correspondence
  between functions, graphs and epigraphs (or hypographs). Let
  $\map{\alpha_1}{\realnonnegative}{\real}$ be the function whose
  epigraph is the convex hull of the epigraph of $\alpha$, i.e.,
  $\epi{\alpha_1}\triangleq\co({\epi{\alpha}})$.  Thus, $\alpha_1$ is
  convex, nondecreasing, and $0\le\alpha_1(s)\le\alpha(s)$ for all
  $s\ge 0$ because
  $\realnonnegative\times\realnonnegative\supseteq\epi{\alpha_1}=\co({\epi{\alpha}})\supseteq\epi{\alpha}$. Moreover,
  $\alpha_1$ is continuous in $(0,\infty)$ by
  convexity~\cite[Th. 10.4]{RTR:70}, and is also continuous at~$0$ by
  the sandwich theorem~\cite[p. 107]{JL-ML:88} because
  $\alpha\in\classkinfty$.  To show that $\alpha_1\in\classkinfty$, we
  have to check that it is unbounded, positive definite in
  $\realnonnegative$, and strictly increasing. First, since
  $s^2/\alpha(s)\in \Obig(s)$ as $s\to\infty$, there exist constants
  $c_1, s_0>0$ such that $\alpha(s)\ge c_1 s$ for all $s> s_0$. Now,
  define $g_1(s)\triangleq\alpha(s)$ if $s\le s_0$ and
  $g_1(s)\triangleq c_1 s$ if $s> s_0$, and $g_2(s)\triangleq -c_1 s_0
  + c_1s$ for all $s\ge 0$, so that $g_2\le g_1\le\alpha$.  Then,
  $\epi{\alpha_1}=\co({\epi{\alpha}})\subseteq\co(\epi{g_1})\subseteq\epi{g_2}$,
  because $\epi{g_2}$ is convex, and thus $\alpha_1$ is unbounded.
  Also, since
  $\co(\epi{g_1})\,\cap\,\realnonnegative\times\{0\}=\{(0,0)\}$, it
  follows that $\alpha_1$ is positive definite.
  To show that $\alpha_1$ is strictly increasing, we use two facts:
  since $\alpha_1$ is convex, we know that the set in which $\alpha_1$
  is allowed to be constant must be of the form $[0,b]$ for some
  $b>0$; on the other hand, since $\alpha_1$ is positive definite, it
  is nonconstant in any neighborhood of~$0$. As a result, $\alpha_1$
  is nonconstant in any subset of its domain, so it is strictly
  increasing.

  Next, let $\map{\alpha_2}{\realnonnegative}{\real}$ be the function
  whose hypograph is the convex hull of the hypograph of $\alpha$,
  i.e., $\hyp{\alpha_2}\triangleq\co({\hyp{\alpha}})$. The function
  $\alpha_2$ is well-defined because $\alpha(s)\in\Obig(s)$ as
  $s\to\infty$, i.e., there exist constants $c_2, s_0>0$ such that
  $\alpha(s)\le c_2 s$ for all $s> s_0$, so if we define
  $g(s)\triangleq c_2 s_0 + c_2 s$ for all $s\ge 0$, then
  $\hyp{\alpha_2}=\co({\hyp{\alpha}})\subseteq\hyp{g}$, because
  $\hyp{g}$ is convex, and thus $\alpha_2(s)\le g(s)$.  Also, by
  construction, $\alpha_2$ is concave, nondecreasing, and
  $\alpha_2\ge\alpha$ because $\hyp{\alpha_2}\supseteq\hyp{\alpha}$,
  which also implies that $\alpha_2$ is unbounded.  Moreover,
  $\alpha_2$ is continuous in $(0,\infty)$ by
  concavity~\cite[Th. 10.4]{RTR:70}, and is also continuous at~$0$
  because the possibility of an infinite jump is excluded by the fact
  that $\alpha_2\le g$.
  To show that $\alpha_2\in\classkinfty$, we have to check that it is
  positive definite in $\realnonnegative$ and strictly
  increasing. Note that $\alpha_2$ is positive definite because
  $\alpha_2(0)=0$ and $\alpha_2\ge\alpha$.  To show that $\alpha_2$ is
  strictly increasing, we reason by contradiction. Assume that
  $\alpha_2$ is constant in some closed interval of the form $[s_1,
  s_2]$, for some $s_2>s_1\ge 0$. Then, as $\alpha_2$ is concave, we
  conclude that it is nonincreasing in $(s_2,\infty)$. Now, since
  $\alpha_2$ is continuous, we reach the contradiction that
  $\lim_{s\to\infty}\alpha(s) \le
  \lim_{s\to\infty}\alpha_2(s)\le\alpha_2(s_1)<\infty$. Hence,
  $\alpha_2$ is strictly increasing.
\end{proof}

\end{document}